\documentclass[11pt,reqno,a4paper]{amsart}
\usepackage{amsmath,amssymb,amsfonts}
\usepackage{amsthm}
\usepackage{array}
\usepackage{resizegather}
\usepackage{array,longtable}
\newcolumntype{C}{>{$}c<{$}}
\newcolumntype{L}{>{$}l<{$}}
\usepackage{color}
\definecolor{myorange}{RGB}{235,140,30}
\definecolor{mygreen}{RGB}{28,172,0} 
\definecolor{mylilas}{RGB}{170,55,241}
\definecolor{azulescuro}{RGB}{0,0,102}
\usepackage[utf8]{inputenc}
\usepackage{graphicx}
\usepackage{setspace}
\usepackage{epstopdf}
\usepackage{soul}
\usepackage{amsmath}
\usepackage{amsfonts}
\usepackage{stackengine}
\usepackage{amssymb}
\usepackage{amsthm}
\usepackage{verbatim}
\usepackage{array}
\usepackage{listings}
\usepackage{float}
\usepackage{mathtools}
\usepackage{xcolor, soul}
\usepackage{relsize}
\usepackage{enumitem}
\usepackage{booktabs, isotope}
\usepackage{caption}
\usepackage{subcaption}
\usepackage{wrapfig}
\allowdisplaybreaks
\usepackage[toc,page]{appendix}
\usepackage{multirow}
\usepackage{lmodern}
\usepackage[hidelinks]{hyperref}
\mathtoolsset{showonlyrefs}

\theoremstyle{plain}
\newtheorem{thm}{Theorem}[section]

\newtheorem{lemma}[thm]{Lemma}
\newtheorem{proposition}[thm]{Proposition}
\newtheorem{corollary}[thm]{Corollary}

\theoremstyle{definition}

\newtheorem{remark}{Remark}[section]

\numberwithin{equation}{section}

\newcommand{\ds}{\displaystyle}
\newcommand{\dint}{\ds\int}
\newcommand{\dsum}{\ds\sum}

\newcommand{\eqskip}{ \vspace*{2mm}\\ }
\newcommand{\R}{\mathbb{R}}
\newcommand{\N}{\mathbb{N}}

\newcommand{\res}{\mathop{\mathrm{Res}} }

\newcommand{\fr}[2]{\frac{\ds #1}{\ds #2}}

\newcommand{\re}{{\rm Re}}
\newcommand{\sn}[1]{\mathbb{S}^{#1}}

\newcommand{\z}[1]{\zeta_{\sn{n}}(#1)}
\newcommand{\zs}[2]{\zeta_{\sn{n}}(#1,#2)}
\newcommand{\zdim}[2]{\zeta_{\sn{#2}}(#1)}
\newcommand{\zsdim}[3]{\zeta_{\sn{#3}}(#1,#2)}
\newcommand{\zder}[1]{\zeta'_{\sn{n}}(#1)}
\newcommand{\zderdim}[2]{\zeta'_{\sn{#2}}(#1)}
\newcommand{\zsder}[2]{\zeta'_{\sn{n}}(#1,#2)}
\newcommand{\zsderdim}[3]{\zeta'_{\sn{#3}}(#1,#2)}

\newcommand{\spec}[2]{\omega_{#1 , #2}}
\newcommand{\specshift}[2]{\lambda_{#1 , #2}}
\newcommand{\lambdas}[1]{\lambda_{#1}}

\newcommand{\snh}[1]{\mathbb{S}^{#1}_+}
\newcommand{\zsh}[2]{\zeta_{\snh{n}}(#1,#2)}
\newcommand{\zsdimh}[3]{\zeta_{\snh{#3}}(#1,#2)}
\newcommand{\zsderh}[2]{\zeta'_{\snh{n}}(#1,#2)}
\newcommand{\zderdimh}[2]{\zeta'_{\snh{#2}}(#1)}
\newcommand{\zshderdim}[3]{\zeta'_{\snh{#3}}(#1,#2)}

\newcommand{\zsrp}[2]{\zeta_{\snrp{n}}(#1,#2)}
\newcommand{\snrp}[1]{\mathbb{RP}^{#1}}
\newcommand{\zsdimrp}[3]{\zeta_{\snrp{#3}}(#1,#2)}

\newcommand{\zderdimrp}[2]{\zeta'_{\snrp{#2}}(#1)}

\newcommand{\Hh}[1]{H^{#1}}
\newcommand{\zH}[2]{\zeta_{\Hh{#2}}(#1)}

\newcommand{\norm}[1]{\left\lVert#1\right\rVert}
\newcommand{\abs}[1]{\left\lvert#1\right\rvert}

\makeatletter
\@namedef{subjclassname@2020}{%
	\textup{2020} Mathematics Subject Classification}
\makeatother

\usepackage{etoolbox}
\patchcmd{\section}{\scshape}{\bfseries}{}{}
\makeatletter
\renewcommand{\@secnumfont}{\bfseries}
\makeatother

\begin{document}

\title{Recurrence formulae for spectral determinants}

\author[J. Cunha]{Jos\'{e} Cunha}
\author[P. Freitas]{Pedro Freitas}

\address{Departamento de Matem\'{a}tica, Instituto Superior T\'{e}cnico, Universidade de Lisboa, Av. Rovisco Pais, 1049-001 Lisboa, Portugal}
\email{jose.d.a.cunha@tecnico.ulisboa.pt}
\address{Departamento de Matem\'{a}tica, Instituto Superior T\'{e}cnico, Universidade de Lisboa, Av. Rovisco Pais, 1049-001 Lisboa, Portugal
	\& Grupo de F\'{\i}sica Matem\'{a}tica, Faculdade de Ci\^{e}ncias, Universidade de Lisboa, Campo Grande, Edif\'{\i}cio C6,
	1749-016 Lisboa, Portugal}
\email{pedrodefreitas@tecnico.ulisboa.pt}

\date{\today}

\begin{abstract} 
	We develop a unified method to study spectral determinants for several different manifolds, including spheres and 
	hemispheres, and projective spaces. This is a direct consequence of an approach based on deriving recursion relations for the 
	corresponding zeta functions, which we are then able to solve explicitly. Apart from new applications such as
	hemispheres, we also believe that the resulting formulae in the cases for which expressions for the determinant
	were already known are simpler and easier to compute in general, when compared to those resulting from other approaches.
\end{abstract}
\maketitle
\tableofcontents 

\section{Introduction}

In the last few decades the problem of evaluating the determinant of the Laplace operator on Riemannian manifolds has received considerable attention in the literature.
These calculations may be traced back to the work of Minakshisundaram and Pleijel~\cite{minaplei}, Seeley~\cite{seel}, and Ray and Singer~\cite{raysing}, and they are
based on a regularization procedure via an associated zeta function. Techniques related to zeta function regularization gained some momentum within the mathematical physics
community after Dowker and Critchley's~\cite{dowkcrit}, and Hawking's~\cite{hawk} papers from 1976 and 1977, respectively, described possible applications to physics~--~for
a more complete historical account, see~\cite{eliz}, for instance; see also~\cite{gelfyagl}. 

Briefly, given an elliptic differential operator $\mathfrak{T}$ defined on a compact manifold (with or without boundary) with discrete spectrum
$\lambda_{1}\leq \lambda_{2} \leq \lambda_{2}\leq \dots$, where in the case without boundary we leave out the zero eigenvalue, we define the spectral
zeta function associated with $\mathfrak{T}$ by
\[
 \zeta_{\mathfrak{T},M}(s)=\sum_{k=1}^{\infty}\left(\lambdas{k}\right)^{-s},
\]
which converges on some half--plane $\Re(s)>\mu$. Under certain conditions, which will be satisfied by the Laplace--Beltrami operator on a compact manifold $M$ (with or
without boundary), for instance, the function $\zeta_{\mathfrak{T},M}$ may be continued analytically to a meromorphic function on the whole of the complex plane, while being analytic
at zero. We now see that, if it were possible to differentiate the original series with respect to $s$ at zero we would obtain, formally,
\[
 \zeta_{\mathfrak{T},M}'(0)= -\sum_{k=1}^{\infty}\log\left(\lambda_{k}\right),
\]
and thus
\begin{equation}\label{infprod}
 e^{-\zeta_{\mathfrak{T},M}'(0)} = \prod_{k=1}^{\infty} \lambda_{k}.
\end{equation}
Clearly this procedure is not justified as such, but it suggests that we may use the analytic continuation of $\zeta_{\mathfrak{T},M}$ to the whole complex plane
as a meromorphic function to define the determinant of the operator $\mathfrak{T}$ as
\begin{equation}\label{def:det:exp}
	\det(\mathfrak{T},M) := e^{-\zeta'_{\mathfrak{T},M}(0)},
\end{equation}
where we now use the expression $\zeta_{\mathfrak{T},M}$ to denote this meromorphic function. Also, whenever there is no ambiguity, we will omit the operator 
both in the index of the zeta function and in the determinant as in~\eqref{def:det:exp}.

Formally, the above definition of $\det(\mathfrak{T},M)$ is the product of the nonzero eigenvalues of the operator $\mathfrak{T}$ acting
on $M$, and is, in fact, a natural extension to the infinite dimensional setting of a formula that is valid in finite dimensions. While this allows us to make
sense of the infinite product appearing in~\eqref{infprod}, the expression on the right-hand side of~\eqref{def:det:exp} will not, in general, admit
a straightforward evaluation. In spite of this, there are several manifolds where this has been done, and for which the determinant has been calculated explicitly.
One such example which will be relevant for us are the $n-$spheres $\mathbb{S}^{n}$ with the standard metric, which have, in fact, deserved
the attention of many authors across a span of more than thirty years, beginning with the work of Vardi in 1988~\cite{vardi}, and followed by several
others such as Voros~\cite{voros}, Quine and Choi~\cite{quine2}, Kumagai~\cite{kumagai}, Quine, Heydari and Song~\cite{quine1}, Choi and
Srivastava~\cite{choisri1,choisri2}, Awonusika~\cite{Awonusika} and Halji~\cite{Halji}, with the last two papers dating from $2020$.

As a first simple example, consider the case of the circle $\mathbb{S}^{1}$, for which, leaving out the zero eigenvalue, the spectrum is given
by $\lambda_{1,k}=k^2$ for $k=1,2\dots$, with multiplicities $m_k^1=2$. This yields an associated zeta function given by 
\begin{equation*}
	\zdim{s}{1}=\sum_{k=1}^{\infty}\dfrac{2}{k^{2s}}=2\zeta(2s),
\end{equation*}  
where $\zeta(s)$ denotes the Riemann zeta function. Proceeding as described above and taking into consideration that $\zeta'(0) =-\frac{1}{2} \log (2 \pi )$,
we obtain $\det(\Delta,\sn{1})=4\pi^2$. It is possible to continue in a similar way for higher-dimensional spheres, but not only do the resulting expressions
become more involved as $n$ increases, it also becomes clear that in order to find a general formula it will be necessary to make the dependence on the dimension
more explicit -- for illustration, see the expressions for the determinant for dimensions up to $9$, given in Corollary~\ref{lowdimSn}. This has led the authors
of the papers quoted above to try out different procedures to obtain closed-form formulae valid for all $n$, with the first explicit results having
been obtained by Quine and Choi~\cite{quine2}, and Kumagai~\cite{kumagai}. Finally, we also note that spheres with the standard metric play an important role as
critical and extremal metrics for the determinant of the Laplacian among certain classes of metrics~\cite{okik,osgood}.

The purpose of the present paper is to provide a unified approach allowing us not only to compute the determinant of $\mathbb{S}^n$ but also of other manifolds
such as hemispheres and projective spaces. This is based on shifting the original eigenvalue sequence in order to obtain a \emph{suitable} eigenvalue sequence leading
to an appropriate zeta function to which we can then apply the method developed by the first author for the case of the quantum harmonic oscillator~\cite{frei}, namely,
the derivation of a recursion formula for the new zeta function. This is then related to the zeta function of the original problem using the techniques devised by
Voros~\cite{voros} -- this part of the process has some similarities with Choi's paper~\cite{choioddsphere}, where a general formula for the determinant in the odd-dimensional
sphere is given. An important step in our method is an explicit formula for the spectral zeta function associated with the shifted eigenvalue sequence (see
Lemma \ref{zetarecursionlemma}), obtained as a result of being able to determine the solutions of the recursion equations satisfied by the zeta-function.
To the best of our knowledge, the derivation of recursion formulae and the solution of the corresponding equations in the context of spectral determinants had
not been used previously, although we could trace some examples of the derivation of recursion formulae for partition functions to the work of Camporesi~\cite{camp}.
The recursions obtained also emphasize the dependence of determinants on the dimension. This is a well-known feature, which has been observed on several occasions
in the literature -- see, for instance~\cite{bar,frei} for the determinant of the Dirac operator on the $n$-dimensional sphere and the quantum harmonic oscillator in
$n$ dimension, and~\cite{okik} for extremal problems on spheres.

As a result of this process, we obtain explicit expressions for the determinant of the Laplacian on even- and odd-dimensional spheres, which may
be found in Theorems~\ref{zetaodddet} and~\ref{zetaevendet}, respectively. For further reference and comparison purposes, and apart from providing explicit expressions for
some low--dimensional spheres explicitly, we also compute the numerical values up to dimension $10\ 000$, showing the first $100$ in Table~\ref{computation:det}.

Due to the flexibility of our method, we are able to apply it to other eigenvalue sequences such as those corresponding to hemispheres (with Dirichlet boundary conditions),
and real projective spaces, for instance. The latter case was studied recently in~\cite{HartSpre}, but we believe the results for hemispheres to be new --
see~\cite{HartSpre2} for a study of the analytic torsion in that case. The graphs with the values of the determinants as a function of the dimension in all these cases
may be seen in Figures~\ref{gr:sphere1},~\ref{gr:hemisphere1}, and ~\ref{gr:proj}, displaying in a very clear way the dependence of these values
on whether the dimension is even or odd.  The behaviour observed in these graphs raises several questions related to monotonicity, the precise asymptotic behaviour
of the determinants in each case and, on a more speculative side, whether the different limits observed in the only case considered where the manifold has a boundary
are a consequence of that fact.

Due to its nature, we believe this method to be applicable to other situations such as complex and quaternionic projective spaces, and to hemispheres with Neumann
boundary conditions, among others. These also include the quantum harmonic oscillator, closely related to the Dirac operator on spheres, which was
studied in~\cite{frei}, and to which we return now to determine a closed-form solution of the recursion formulae derived there for the corresponding zeta function.
The fact that in this case the asymptotic behaviour of the determinant in the dimension is exponentially decreasing, while we expect the remaining cases analysed here to
have an algebraic behaviour, would indicate that our approach is not restricted to determinants with a specific behaviour.


The structure of the paper is as follows. In the next section we collect the necessary results and describe the procedure using the case of the $n-$sphere as an example. In
Section~\ref{otherex} we apply our method to the other examples already mentioned above. Appendix~\ref{appendixcfn} contains some useful facts about central factorial numbers which are
used throughout the paper, while in Appendix~\ref{numerical} we provide some tables with numerical values for reference.

\section{Method description: the case of $\sn{n}$}

\subsection{Eigenvalues and a tale of two zeta functions}

Let $\sn{n}=\{x\in \mathbb{R}^{n+1}\; : \; \norm{x}=1\}$ be the $n-$dimensional unit sphere with the standard metric induced by the $\mathbb{R}^{n+1}$ Euclidean norm. The spectrum
of the Laplace operator $\Delta$ on $\sn{n}$ is well known and given by~\cite{bang}
\begin{align*}
	\spec{n}{k}&=k(k+n-1)
\end{align*}
for $k=0,1,2,\dots,$ with multiplicities
\begin{equation}
	m_k^n:=\text{mult}(\spec{n}{k})=\binom{n+k}{k}-\binom{n+k-2}{k-2}=(2k+n-1)\dfrac{(k+n-2)!}{k!(n-1)!} \label{multi}
\end{equation}
Consider now the associated zeta function
\begin{align*}
	\z{s}&=\sum_{k=1}^{\infty} \left(\spec{n}{k}\right)^{-s}=\sum_{k=1}^{\infty} \dfrac{m^n_k}{(k(k+n-1))^s}
\end{align*}
In order to compute $\zder{0}$, we will consider a shift of the eigenvalues by a constant yielding a perfect square and a more manageable zeta function. To this end,
define
\begin{align*}
	\specshift{n}{k}:=\spec{n}{k}+\lambdas{n}=\left(k+\frac{n-1}{2}\right)^2,
\end{align*}
the eigenvalue sequence obtained after shifting $\omega_{n,k}$ by $\lambdas{n}=\left(\tfrac{n-1}{2}\right)^2$, and consider its associated zeta function
\begin{align}
	\zs{s}{\lambdas{n}}=\sum_{k=1}^{\infty}\left(\spec{n}{k}+\lambdas{n}\right)^{-s}=\sum_{k=1}^{\infty}\left(\specshift{n}{k}\right)^{-s} 
	\label{shitfzeta}
\end{align}
It is clear that $\zs{s}{0}=\z{s}$ and, for simplicity, in what follows we use a prime to denote the derivative with respect to the variable $s$, namely,
$\zeta'(s,a)=\frac{\partial\zeta(s,a)}{\partial s}$. More important, this new zeta function yields an associated determinant which is now possible to calculate in any
dimension, and from which the original determinant may then be retrieved.

\subsection{Relating the determinants of $\z{s}$ and $\zs{s}{\lambdas{n}}$: the method of Voros}\hfill\\

In 1987 Voros studied several functions associated with infinite increasing sequences of real numbers~\cite{voros}. This was later
extended to the more general case of complex number sequences by Quine, Heydari and Song~\cite{quine1}. Of particular relevance here are sequences formed
by the eigenvalues of elliptic operators, including the Laplacian, and one important issue is to know when a certain sequence may be what is referred to
in~\cite{quine1} as {\it zeta regularizable}, that is, when the corresponding zeta function has a meromorphic continuation with at most
simple poles, to a right-half-plane containing the origin, and is analytic at the origin.

One possible way to address these issues, which was used in~\cite{voros}, uses techniques from analytic number theory, including the
analytical continuation of Mellin transforms and the Weierstrass canonical product $E(\lambda)$ associated with the sequence $\{\lambda_k\}$, defined by
\begin{equation}
	E(\lambda):=\prod_{k=1}^{\infty}\left\{ \left( 1-\dfrac{\lambda}{\lambda_k}\right)\exp\left(\dfrac{\lambda}{\lambda_k}+\dfrac{\lambda^2}{2\lambda_k^2}+\dots+
	\dfrac{\lambda^{\lfloor\mu\rfloor}}{\lfloor\mu\rfloor\lambda_k^{\lfloor\mu\rfloor}}\right)\right\},  \label{weiestrass}
\end{equation}
where $\mu$ denotes the abscissa of (absolute) convergence of $\zeta_{\mathfrak{T},M}$, and the sum in the exponent is considered to vanish when $\mu<1$.

In the following theorem we collect a result from~\cite[pp. 447]{voros} in an appropriate form to be used by us in the sequel.
\begin{thm}[Voros~\cite{voros}]
 Suppose that the sequence of eigenvalues $\{\lambda_k\}$ associated with the operator $\mathfrak{T}$ acting on $M$ is a monotonically increasing sequence of real numbers.
 Given a constant $\lambda$, the relationship between the determinants associated with the sequences $\{\lambda_k-\lambda\}$ and $\{\lambda_k\}$ is given by
\begin{equation*}
		\label{voroseq}
		\zeta'_{\mathfrak{T},M}(0)=  \zeta'_{\mathfrak{T},M}(0,-\lambda)
		+\sum_{m=1}^{\lfloor\mu\rfloor}FP\left[\zeta_{\mathfrak{T},M}(m,-\lambda)\right]\frac{\lambda^m}{m}
		+\sum_{m=2}^{\lfloor\mu\rfloor}c_{-m}H_{m-1}\frac{\lambda^m}{m!}-\log(E(\lambda))
\end{equation*}
where an empty sum (i.e. the case where $\mu < 1$) is to be considered zero, the Harmonic numbers $H_{n}$ are given by
\begin{equation*}
	H_n =\sum_{k=1}^{n}\dfrac{1}{k},
\end{equation*}
the finite part $(FP)$ is defined by
\begin{equation} 
	FP[f(s)] = 
	\begin{cases} 
		f(s),   & \quad \text{if} \;s\; \text{is not a pole of }f\\
		\lim\limits_{\epsilon\rightarrow0}\left(f(s+\epsilon)-\dfrac{\res(f,s)}{\epsilon}\right), & \quad \text{if} \;s\; \mbox{is a pole of } f
	\end{cases}
	\label{fp}
\end{equation} 
and 
\begin{equation}
	c_{-m}=\res(\zeta_{\mathfrak{T},M},m)\Gamma(m)
	\label{Cm}
\end{equation}
\label{voros} 
where $\Gamma(z)$ denotes the Gamma function.
\end{thm}
\begin{remark}
 For a different formula for the coefficients $c_{-m}$ see Section~\ref{disc}.
\end{remark}

For our purposes, and for ease of reference, consider the following corollary of Theorem~\ref{voros}.
\begin{corollary} Assume that $\abs{\tfrac{\lambda}{\lambda_k}}<1$ for all $k\in\N$. Then
	\begin{align*}
			\zeta'_{\mathfrak{T},M}(0)&=  \zeta'_{\mathfrak{T},M}(0,-\lambda)+\sum_{m=1}^{\lfloor\mu\rfloor}FP\left[\zeta_{\mathfrak{T},M}(m,-\lambda)\right]\frac{\lambda^m}{m}+\sum_{m=2}^{\lfloor\mu\rfloor}c_{-m}H_{m-1}\frac{\lambda^m}{m!}\\
			&\hspace{5mm}+\sum_{m=\lfloor\mu\rfloor+1}^{\infty} \zeta_{\mathfrak{T},M}(m,-\lambda)\dfrac{\lambda^m}{m}
	\end{align*}
	\label{voroscorollary}
\end{corollary}
\begin{proof}
	Applying logarithms to both sides of \eqref{weiestrass} yields
	\begin{align*}
		-\log(E(\lambda))&= \sum_{k=1}^{\infty} \left[- \log\left(1-\dfrac{\lambda}{\lambda_k}\right) - \sum_{m=1}^{\lfloor\mu\rfloor} \dfrac{\lambda^m}{m\lambda_k^m}\right] \eqskip
		&= \sum_{k=1}^{\infty} \left[ \sum_{m=1}^{\infty}\dfrac{\lambda^m}{m\lambda_k^m} - \sum_{m=1}^{\lfloor\mu\rfloor} \dfrac{\lambda^m}{m\lambda_k^m}\right]  \eqskip
		&=\sum_{m=\lfloor\mu\rfloor + 1}^{\infty} \zeta_{\mathfrak{T},M}(m,-\lambda)\dfrac{\lambda^m}{m} 
	\end{align*}
	where, in the first step, we can exchange the order of summation of the infinite series since both series are absolutely convergent due to the assumption
	$\abs{\tfrac{\lambda}{\lambda_k}}<1$ for all $k$; this also allows us to use the logarithm series expansion in the second line.
\end{proof}

\subsection{Recursions}

We shall now introduce the crucial step underlying our approach. Following Theorem~A in~\cite{frei}, in which a two-term recursion formula in the dimension was derived
for the spectral determinant of the quantum harmonic oscillator, we show that a similar recursion may also be obtained in our first example, i.e the Laplace operator on $\sn{n}$.
For this case, the zeta function $\zs{s}{\lambdas{n}}$ defined in~\eqref{shitfzeta} takes the form
\begin{equation}
	\label{zetaboring}
	\zs{s}{\lambdas{n}}=\sum_{k=1}^{\infty} \dfrac{m_k^n}{\left(k+\frac{n-1}{2}\right)^{2s}}.
\end{equation}
\begin{lemma} \label{zetarecursionlemma} The zeta function $\zs{s}{\lambdas{n}}$ in \eqref{zetaboring} satisfies the following two-term recursion
	\begin{equation*}
		\zsdim{s}{\lambdas{n+2}}{n+2}=\dfrac{\zs{s-1}{\lambdas{n}}-(\frac{n-1}{2})^2 \zs{s}{\lambdas{n}}}{n(n+1)}-\left(\frac{2}{n+1}\right)^{2s}
	\end{equation*}
	with initial conditions
	\begin{align*}
		\zsdim{s}{\lambdas{1}}{1}&=2\zeta(2s) \\
		\zsdim{s}{\lambdas{2}}{2}&=(4^s-2)\zeta(2s-1)-4^s 
	\end{align*}
\end{lemma}
\begin{proof}
	Consider the multiplicities $m^n_k$ given in~\eqref{multi} and note that
	\begin{align*}
		m^{n+2}_k&=(2k+n+1)\dfrac{(k+n)!}{k!(n+1)!}\\
		&=m^{n}_{k+1} \dfrac{(k+n)(k+1)}{n(n+1)}\\
		&=m^{n}_{k+1} \dfrac{k^2+k(n+1)+n}{n(n+1)}\\
		&=m^{n}_{k+1} \dfrac{(2k+n+1)^2-(n-1)^2}{4n(n+1)}.
	\end{align*}
	Substituting this in the expression for $\zs{s}{\lambdas{n}}$ given by~\eqref{zetaboring} yields
	\begin{align*}
		\zsdim{s}{\lambdas{n+2}}{n+2}&=\mathlarger{\sum}_{k=1}^{\infty} \dfrac{m_k^{n+2}}{\left(k+\frac{n+1}{2}\right)^{2s}}\eqskip
		&=\mathlarger{\sum}_{k=1}^{\infty}\frac{m_{k+1}^{n}}{\left(k+\tfrac{n+1}{2}\right)^{2s}}\dfrac{(2k+n+1)^2-(n-1)^2}{4n(n+1)}\eqskip
		&=\mathlarger{\sum}_{k=1}^{\infty}\frac{m_{k+1}^{n}}{\left(k+\tfrac{n+1}{2}\right)^{2s}}\dfrac{(k+\frac{n+1}{2})^2-(\frac{n-1}{2})^2}{n(n+1)}\eqskip
		&=\left[\mathlarger{\sum}_{k=1}^{\infty}\frac{m_{k+1}^{n}}{\left(k+\tfrac{n+1}{2}\right)^{2(s-1)}}-\left(\frac{n-1}{2}\right)^2 
		\mathlarger{\sum}_{k=1}^{\infty}\frac{m_{k+1}^{n}}{\left(k+\tfrac{n+1}{2}\right)^{2s}}\right]\frac{1}{n(n+1)}\eqskip
		&=\left[\mathlarger{\sum}_{k=2}^{\infty}\frac{m_{k}^{n}}{\left(k+\tfrac{n-1}{2}\right)^{2(s-1)}}-\left(\frac{n-1}{2}\right)^2
		\mathlarger{\sum}_{k=2}^{\infty}\frac{m_{k}^{n}}{\left(k+\tfrac{n-1}{2}\right)^{2s}}\right]\frac{1}{n(n+1)}\eqskip
		&=\dfrac{\zs{s-1}{\lambdas{n}}-(\frac{n-1}{2})^2 \zs{s}{\lambdas{n}}}{n(n+1)}-\left(\frac{2}{n+1}\right)^{2s}.
	\end{align*}
\end{proof}
Even though $\{\specshift{n}{k}\}$ is not the original eigenvalue sequence associated with the Laplacian on $\sn{n}$, its relevance should be clear by now. In the preceding
section we presented a method due to Voros by which we are able to relate the functional determinant of zeta functions associated with different sequences, as long as these are only
shifted by a constant. In fact we have that 
\begin{align*}
	 -\log\left[ \det\left(\sn{n}\right)\right]&=\zderdim{0}{n}=\zsder{0}{\lambdas{n}}+\sum_{m=1}^{\lfloor\mu\rfloor}FP\left[\zs{m}{\lambdas{n}}\right]\frac{(\lambdas{n})^m}{m}+\sum_{m=2}^{\lfloor\mu\rfloor}c_{-m}H_{m-1}\frac{(\lambdas{n})^m}{m!}\\
	&\hspace{22mm}+\sum_{m=\lfloor\mu\rfloor+1}^{\infty} \zs{m}{\lambdas{n}}\dfrac{(\lambdas{n})^m}{m}
\end{align*}
The recursion in the preceding lemma will enable us to derive an explicit formula for $\zs{s}{\lambdas{n}}$, specifically, a finite weighted sum of Riemann zeta functions $\zeta(s)$.
The properties of such weights are induced by the recursion and shall be dealt with in Appendix \ref{appendixcfn}. As mentioned in the Introduction, this recursion clearly shows a
dual behavior on the dimension.

\subsection{Solution to the recursion: explicit expressions for $\zs{s}{\lambdas{n}}$}
We are now able to solve the recursion in Lemma \ref{zetarecursionlemma} to obtain more manageable expressions for $\zs{s}{\lambdas{n}}$ in even and odd dimensions separately. 

\begin{thm}
Let $n\in\N$. The zeta function $\zs{s}{\lambdas{n}}$ defined in \eqref{zetaboring} satisfies the following identities
\begin{align*}
	\zsdim{s}{\lambdas{2n+1}}{2n+1}&=\sum_{i=1}^{n}\bar{u}(n,i)\zeta(2(s-i)) - n^{-2s}\\
	\zsdim{s}{\lambdas{2n}}{2n}&=\sum_{i=1}^{n}\bar{v}(n,i)\big(4^s-2^{2i-1}\big)\zeta(2s-2i+1) - \left(\frac{2n-1}{2}\right)^{-2s}
\end{align*}	
	where $\bar{u}(n,i)$ and $\bar{v}(n,i)$ are defined in~\eqref{cfnevennorm} and~\eqref{cfnoddnorm} respectively.
	\label{zetarecsol}
\end{thm}
\begin{proof}\hfill\\
	- Odd dimensional case: 	
	\item $\blacklozenge$ Induction on $n$.\\
	$\diamondsuit$ \emph{Base case}: $n=1$ \\
	In view of Lemma \ref{zetarecursionlemma}, $\zsdim{s}{\lambdas{3}}{3}=\zeta(2(s-1))-1$ which agrees with Theorem \ref{zetarecsol} since $\bar{u}(1,1)=1$.\\
	$\diamondsuit$ \emph{Induction step}:\\ Following the recursion in Lemma \ref{zetarecursionlemma} for $\zsdim{s}{\lambdas{2n+1}}{2n+1}$ we obtain	
	\begin{align*}
		&\zsdim{s}{\lambdas{2n+1}}{2n+1}=\dfrac{\zsdim{s-1}{\lambdas{2n-1}}{2n-1}-(n-1)^2\zsdim{s}{\lambdas{2n-1}}{2n-1}}{2n(2n-1)}-n^{-2s}&&\\ &=\dfrac{\sum\limits_{i=1}^{n-1}\bar{u}(n-1,i)\zeta(2(s-i-1))-(n-1)^2\sum\limits_{i=1}^{n-1}\bar{u}(n-1,i)\zeta(2(s-i))}{2n(2n-1)}-n^{-2s}&& \small\text{Induction Hypothesis}\\
		&=\dfrac{\sum\limits_{i=1}^{n}\bar{u}(n-1,i-1)\zeta(2(s-i)) -(n-1)^2\sum\limits_{i=1}^{n}\bar{u}(n-1,i)\zeta(2(s-i))}{2n(2n-1)}-n^{-2s}&& \small\text{Proposition \ref{cfnprop} $\mathit{(iv)}$ and $\mathit{(v)}$}\\  
		&=\left[\sum\limits_{i=1}^{n}\frac{\bar{u}(n-1,i-1)-(n-1)^2\bar{u}(n-1,i)}{2n(2n-1)}\zeta(2(s-i))\right] -n^{-2s}&&\\	 
		&=\sum\limits_{i=1}^{n}\bar{u}(n,i)\zeta(2(s-i)) -n^{-2s}&& \small\text{Proposition \ref{cfnevenprop} $\mathit{(ii)}$}	   
	\end{align*}
    - Even dimensional case:
    Analogous to the odd dimensional case by means of Proposition \ref{cfnoddprop}.
\end{proof}
From Theorem \ref{zetarecsol} we conclude that $\zsdim{s}{\lambdas{2n+1}}{2n+1}$ is defined for $\Re(s)> n+\frac{1}{2}$ and it can be meromorphically continued to a
function of the whole complex plane, since it depends uniquely on the zeta function $\zeta(s)$, with exactly $n$ poles at $\{\frac{3}{2},\frac{5}{2},\dots,\frac{2n+1}{2}\}$.
Analogously, we conclude that $\zsdim{s}{\lambdas{2n}}{2n}$ is defined for $\Re(s)> n$ and can be analytically continued to a meromorphic function of the whole complex plane
with exactly $n$ poles at $\{1,2,\dots,n\}$.

\subsection{Recovering $\z{s}$ and the determinant}

We are now able to combine the expressions obtained in Theorem~\ref{zetarecsol} and Voros method, namely, Corollary~\ref{voroscorollary}, to obtain simple and efficiently
computable expressions for the determinant of the Laplacian $\Delta$ on the odd ($\sn{2n+1}$) and even ($\sn{2n}$) dimensional spheres separately.

In the following proposition we recall some properties of the Riemann and Hurwitz zeta functions, $\zeta(s)$ and $\zeta(s,a)$, respectively, and of the Bernoulli numbers $B_n$,
all of which are well-known and may be found in~\cite[Chapter~2]{choisri1}.
\begin{proposition}  Let $p,n\in\N$ and consider the Bernoulli numbers, $B_n$. The following identities hold true.
		
$ \arraycolsep=3pt\def\arraystretch{2.1}
	\begin{array}{cl}
		\mathit{(i)} & \zeta(s)=\zeta(s,1)=\dfrac{1}{(2^s - 1)}\zeta(s,\frac{1}{2}) \\
		\mathit{(ii)} & \zeta(s,a)=\zeta(s,a+n)+\sum\limits_{j=0}^{n-1}(j+a)^{-s}  \\
		\mathit{(iii)} & \zeta'(s,n)=\zeta'(s)+\sum\limits_{j=1}^{n-1}j^{-s}\log(j)  \\
		\mathit{(iv)} & \zeta'(s,n+\frac{1}{2})=\log(2)2^{s}\zeta(s)+(2^{s}-1)\zeta'(s)+\sum\limits_{j=1}^{n}\left(\dfrac{2j-1}{2}\right)^{-s}\log{\left(\frac{2j-1}{2}\right)} \\
		\mathit{(v)} & \zeta'(0)=-\frac{1}{2}\log(2\pi) \\
		\mathit{(vi)} & \zeta(-n)=(-1)^n \dfrac{B_{n+1}}{n+1} \\
		\mathit{(vii)} & \sum\limits_{k=1}^{n}k^p= \dfrac{1}{p+1}\sum\limits_{j=0}^{p}\binom{p+1}{j}B_j n^{p+1-j} \qquad\normalfont\text{Faulhaber's formula} \\
\end{array}	$
	\label{zetaprop}
\end{proposition}
These will be used multiple times throughout. It will also prove useful to consider the following identities which can be found in \cite[pp. 258,  identities~(64) and~(67)]{choisri1}. 
\begin{proposition}\label{crazyformulas} Let $i\in\N_{0}$ and $\abs{t}<\abs{a}$. The following identities hold true.
	\begin{gather*}
	\begin{array}{llll}
		\mathit{(i)}&\;\dsum_{m=1-i}^{\infty}\dfrac{\zeta(2m,a)}{m+i}t^{2m+2i} &=& \dsum_{k=0}^{2i}\binom{2i}{k}\bigg[\zeta'(-k,a-t)+(-1)^k\zeta'(-k,a+t)\bigg]t^{2i-k} - 2\zeta'(-2i,a)\eqskip
		\mathit{(ii)}&\;\dsum_{\substack{m=-i\\i\neq0}}^{\infty}\dfrac{\zeta(2m+1,a)}{m+i+1}t^{2m+2i+2} &=&\dsum_{k=0}^{2i+1}\binom{2i+1}{k}\bigg[\zeta'(-k,a-t)-(-1)^k\zeta'(-k,a+t)\bigg]t^{2i-k+1}\eqskip
		&	&& \hspace{5mm}-\dfrac{t^{2i+2}}{i+1}[\psi(2i+2)-\psi(a)+\gamma]- 2\zeta'(-2i-1,a)
	\end{array}
	\end{gather*}
where $\psi(z)$ denotes the digamma function defined as the logarithmic derivative of $\Gamma$. For $n\in\N$ we have
\begin{equation}
	\psi(n)=H_{n-1}-\gamma
	\label{psiprop}
\end{equation}
where $\gamma$ is the Euler-Mascheroni constant.
\end{proposition}
In particular, these will allow us to compute the Weiestrass canonical product $E(\lambda)$ defined by~\eqref{weiestrass} and reduce it to a finite sum of more manageable functions. 
\subsubsection{The case of $\sn{2n+1}$ }\hfill\\

It is now clear that $FP[\zsdim{m}{\lambdas{2n+1}}{2n+1}]=\zsdim{m}{\lambdas{2n+1}}{2n+1}$ from \eqref{fp} and that $c^{2n+1}_{-m}=0$ from $\eqref{Cm}$ for $m\in\N$. From Corollary \ref{voroscorollary} we obtain
\begin{equation}
	\zderdim{0}{2n+1}=\zsderdim{0}{\lambdas{2n+1}}{2n+1}+\sum_{m=1}^{\infty}\zsdim{m}{\lambdas{2n+1}}{2n+1}\dfrac{(\lambdas{2n+1})^m}{m}
	\label{detpre}
\end{equation}
where $\lambdas{2n+1}=n^2$. Differentiating $\zsdim{s}{\lambdas{2n+1}}{2n+1}$ in Theorem \ref{zetarecsol} at $s=0$ yields
\begin{equation}
	\zsderdim{0}{\lambdas{2n+1}}{2n+1}=\sum_{i=1}^{n}2\bar{u}(n,i)\zeta'(-2i) +\log(n^2)
	\label{zetaoddderivative0}
\end{equation}
In view of \eqref{detpre}, and again using the expression for $\zsdim{s}{\lambdas{2n+1}}{2n+1}$ in Theorem \ref{zetarecsol} consider the following tedious yet simple computations
\begin{align*}
	\hspace{5mm}&\sum_{m=1}^{\infty} \zsdim{m}{\lambdas{2n+1}}{2n+1}\dfrac{(\lambdas{2n+1})^m}{m}=\sum_{m=1}^{\infty}\dfrac{n^{2m}}{m}\left[\sum_{i=1}^{n}\bar{u}(n,i)\zeta(2(m-i)) - n^{-2m}\right] && \\
	&=\sum_{m=1}^{\infty}\dfrac{1}{m}\left[\sum_{i=1}^{n}\bar{u}(n,i)n^{2m}\left(\zeta(2(m-i),n+1)+ \sum\limits_{k=1}^{n}k^{-2m+2i}\right) - 1\right]&&\small\text{Proposition \ref{zetaprop} $\mathit{(ii)}$} \\
	&=\sum_{m=1}^{\infty}\dfrac{1}{m}\Bigg[\sum_{i=1}^{n}\bar{u}(n,i)n^{2m}\zeta(2(m-i),n+1) +\sum_{i=1}^{n}\bar{u}(n,i)n^{2i}&& \\ 
	& \hspace{5mm} +n^{2m}\sum\limits_{k=1}^{n-1}k^{-2m}\sum_{i=1}^{n}\bar{u}(n,i)k^{2i} - 1 \Bigg]&& \\
	&=\sum_{m=1}^{\infty}\dfrac{1}{m}\sum_{i=1}^{n}\bar{u}(n,i)n^{2m}\zeta(2(m-i),n+1)&&\small\text{Proposition \ref{cfnevenprop} $\mathit{(iii)}$ and $\mathit{(iv)}$}\\
	&=\sum_{i=1}^{n}\bar{u}(n,i)\sum_{m=1}^{\infty}\dfrac{\zeta(2(m-i),n+1)n^{2m}}{m}&&\\
	&=\sum_{i=1}^{n}\bar{u}(n,i)\sum_{m=1-i}^{\infty}\dfrac{\zeta(2m,n+1)n^{2m+2i}}{m+i}&&\\
	&=\sum_{i=1}^{n}\bar{u}(n,i)\Bigg(\sum_{k=0}^{2i}\binom{2i}{k}\left[ \zeta'(-k)+(-1)^k\zeta'(-k,2n+1)\right]n^{2i-k} &&  \small\underset{a=n+1, t=n}{\text{Proposition \ref{crazyformulas} $\mathit{(i)}$}}\\
	& \hspace{5mm}-2\zeta'(-2i,n+1)\Bigg)&&\small\text{Proposition \ref{zetaprop} $\mathit{(iii)}$}\\
	&=\sum_{i=1}^{n}\bar{u}(n,i)\Bigg(\sum_{k=0}^{2i}\binom{2i}{k}\zeta'(-k)n^{2i-k}\bigg[\left(1+(-1)^k\right)+(-1)^{k}\sum_{j=1}^{2n}j^k\log(j)\bigg]&& \\
	&\hspace{5mm} - 2\zeta'(-2i)-2\sum_{j=1}^{n}j^{2i}\log(j)\Bigg)&&\\
	&=\sum_{i=1}^{n}\bar{u}(n,i)\Bigg(\sum_{k=0}^{i}2\binom{2i}{2k}\zeta'(-2k)n^{2i-2k} +\sum_{j=1}^{2n}\log(j)n^{2i}\sum_{k=0}^{2i}\binom{2i}{k}\left(\frac{-j}{n}\right)^k && \\
	&\hspace{5mm}-2\zeta'(-2i)+\sum_{j=1}^{n}j^{2i}\log(j^2)\Bigg)&&\\
	&=\sum_{i=1}^{n}\bar{u}(n,i)\Bigg(\sum_{k=0}^{i}2\binom{2i}{2k}\zeta'(-2k)n^{2i-2k}-2\zeta'(-2i)+\sum_{j=1}^{2n}\log(j)(n-j)^{2i} &&\\
	&\hspace{5mm}-\sum_{j=1}^{n}j^{2i}\log(j^2)\Bigg)&& \\
	&=\sum_{i=1}^{n}\bar{u}(n,i)\Bigg(\sum_{k=0}^{i}2\binom{2i}{2k}\zeta'(-2k)n^{2i-2k}-2\zeta'(-2i)+\sum_{j=1}^{n-1}\log\left(1-\frac{j^2}{n^2}\right)j^{2i} && \\
	&\hspace{5mm}-\log\left(\frac{n}{2}\right)n^{2i}\Bigg)&& \\	
	&=\sum_{i=1}^{n}\bar{u}(n,i)\sum_{k=0}^{i}2\binom{2i}{2k}\zeta'(-2k)n^{2i-2k}-\sum_{i=1}^{n}2\bar{u}(n,i)\zeta'(-2i)-\log\left(\frac{n}{2}\right) && \small\text{Proposition \ref{cfnevenprop} $\mathit{(iii)}$ and $\mathit{(iv)}$}
\end{align*}	
\begin{thm} Let $n\in\N$. The determinant of the Laplacian $\Delta$ on odd-dimensional spheres $\sn{2n+1}$ satisfies the following identity
	\begin{equation*}
		-\log\left[ \det\left(\sn{2n+1}\right)\right] = \zderdim{0}{2n+1}=\frac{2}{(2n)!}\sum_{k=1}^{n}\zeta'(-2k)[s(2n,2k)+s(2n+1,2k+1)] + \log\left(\frac{n}{\pi}\right)
		\label{zetaodddet}
	\end{equation*}
	where $s(n,k)$ are the Stirling numbers of the first kind as defined in Proposition \ref{stirlingnumbers}.
\end{thm}
\begin{proof}	
	Using \eqref{detpre}, \eqref{zetaoddderivative0} and the previous result we obtain
	\begin{align*}
		\zderdim{0}{2n+1}&=\sum_{i=1}^{n}2\bar{u}(n,i)\sum_{k=0}^{i}\binom{2i}{2k}\zeta'(-2k)n^{2i-2k}+\log\left(2n\right) \\
		&=\sum_{i=1}^{n}2\bar{u}(n,i)\sum_{k=1}^{i}\binom{2i}{2k}\zeta'(-2k)n^{2i-2k}+2\zeta'(0)\sum_{i=1}^{n}\bar{u}(n,i)n^{2i}+\log\left(2n\right) \\
		&=\sum_{i=1}^{n}2\bar{u}(n,i)\sum_{k=1}^{i}\binom{2i}{2k}\zeta'(-2k)n^{2i-2k}+\log\left(\frac{n}{\pi}\right) \\			&=\sum_{k=1}^{n}2\zeta'(-2k)\sum_{i=k}^{n}\binom{2i}{2k}\bar{u}(n,i)n^{2i-2k}+\log\left(\frac{n}{\pi}\right) \\
		&=\dfrac{2}{(2n)!}\sum_{k=1}^{n}\zeta'(-2k)[s(2n,2k)+s(2n+1,2k+1)] + \log\left(\frac{n}{\pi}\right)	
	\end{align*}
	where the second step follows by Proposition \ref{zetaprop} $\mathit{(v)}$ and Proposition \ref{cfnevenprop} $\mathit{(iii)}$ and the last step follows from Proposition \ref{cfnevenprop} $\mathit{(v)}$.
\end{proof}

\subsubsection{The case of $\sn{2n}$ }\hfill\\

From Corollary \ref{voroscorollary} together with the previous insights we obtain
\begin{align}
	\zderdim{0}{2n}&=\zsderdim{0}{\lambdas{2n}}{2n} +\sum_{m=1}^{n}FP\left[\zsderdim{0}{\lambdas{2n}}{2n}\right]\dfrac{(\lambdas{2n})^m}{m}+\sum_{m=2}^{n}c^{2n}_{-m}H_{m-1}\dfrac{(\lambdas{2n})^m}{m!}\nonumber\\
	&\hspace{5mm}+\sum_{m=n+1}^{\infty} \zsderdim{m}{\lambdas{2n}}{2n}\dfrac{(\lambdas{2n})^m}{m}
	\label{detpre2}
\end{align}
where $\lambdas{2n}=\left(\frac{2n-1}{2}\right)^2$.
Unlike the odd-dimensional case, the components of \eqref{detpre2} need further care. Differentiating $\zsderdim{s}{\lambdas{2n}}{2n}$ in Theorem \ref{zetarecsol} at $s=0$ yields
\begin{equation}
	\zsderdim{0}{\lambdas{2n}}{2n}=\sum_{i=1}^{n}\bar{v}(n,i)\big(\log(4)\zeta(-2i+1)+(2-2^{2i})\zeta'(-2i+1)\big)-2\log\left(\frac{2}{2n-1}\right)
	\label{deteven}
\end{equation}
Consider the coefficients $c^{2n}_{-m}$ defined by~\eqref{Cm} and note that, using Theorem \ref{zetarecsol}, we obtain
\begin{align*}
	c^{2n}_{-m}&= \res\left(\zsdim{s}{\lambdas{2n}}{2n},m\right)\Gamma(m) \\
	&=\big(4^m-2^{2m-1}\big)\bar{v}(n,m)\res\big(\zeta(2s-2m+1),m\big)(m-1)!\\
	&=2^{2m-1}\bar{v}(n,m)\res\big(\zeta(2s-1),1\big)(m-1)!\\
	&=2^{2m-2}\bar{v}(n,m)(m-1)!
\end{align*}
which leads to, in view of~\eqref{detpre2},
\begin{equation}
	\sum_{m=2}^{n}c^{2n}_{-m}H_{m-1}\frac{(\lambdas{2n})^m}{m!}=\sum_{m=2}^{n}\bar{v}(n,m)\frac{(2n-1)^{2m}}{4m}H_{m-1}
	\label{Cmeven}
\end{equation}
Furthermore, consider the definition of the finite part as given by~\eqref{fp} and note that, using Theorem \ref{zetarecsol}, we obtain
\begin{align*}
	FP[\zsdim{m}{\lambdas{2n}}{2n}]&=\sum_{i=1}^{n}\bar{v}(n,i)\big(4^m-2^{2i-1}\big)FP[\zeta(2m-2i+1)] - \left(\frac{2}{2n-1}\right)^{2m}\\
	&=\sum_{\substack{i=1\\i\neq m}}^{n}\underbrace{\bar{v}(n,i)\big(4^m-2^{2i-1}\big)\zeta(2m-2i+1)}_{f_i(m)} - \left(\frac{2}{2n-1}\right)^{2m}\\
	&\quad + \lim_{\epsilon\rightarrow0}\left\{f_m(m+\epsilon) + \frac{\res(f_m,m)}{\epsilon}\right\}
\end{align*}
Using the fact that $\zeta(s)$ satisfies (see \cite[pp. 219]{eulergamma})
\begin{equation*}
	\lim_{\epsilon\rightarrow0}\left\{\zeta(1+\epsilon) - \frac{1}{\epsilon}\right\}=\gamma
	\label{zetalimit}
\end{equation*}
we see that the last term in the previous equation simplifies as follows
\begin{align*}
	\lim_{\epsilon\rightarrow0}\left\{f_m(m+\epsilon) - \frac{\res(f_m,m)}{\epsilon}\right\}&=\lim_{\epsilon\rightarrow0}\left\{ \bar{v}(n,m)\big(2^{2m+2\epsilon}-2^{2m-1}\big)\zeta(1+2\epsilon)-\frac{\bar{v}(n,m)2^{2m-1}}{2\epsilon}\right\}\\
	&=\bar{v}(n,m)2^{2m-1}\lim_{\epsilon\rightarrow0}\left\{ \big(2^{2\epsilon+1}-1\big)\left(\zeta(1+2\epsilon)-\frac{1}{2\epsilon}\right)+\frac{2^{2\epsilon}-1}{\epsilon}\right\}\\
	&=\bar{v}(n,m)2^{2m-1}\big(\gamma+\log(4)\big)
\end{align*}

Consider now, in virtue of~\eqref{detpre2}, the following 
\begin{align*}
	\sum_{m=n+1}^{\infty}\zsdim{m}{\lambdas{2n}}{2n}\frac{(\lambdas{2n})^m}{m}+& \sum_{m=1}^{n}FP[\zsdim{m}{\lambdas{2n}}{2n}]\frac{(\lambdas{2n})^m}{m}=\\
	&=\sum_{m=1}^{\infty}\frac{1}{m}\left[\sum_{\substack{i=1\\i\neq m}}^{n}\bar{v}(n,i)\big(4^m-2^{2i-1}\big)\zeta(2m-2i+1)\left(\frac{2n-1}{2}\right)^{2m}-1\right]\\
	&\hspace{5mm}+ \sum_{i=1}^{n}\bar{v}(n,i)\frac{(2n-1)^{2i}}{2i}\big(\gamma+\log(4)\big)
\end{align*}
We will now handle the infinite sum in the last equation by means of a procedure similar to that used in the odd-dimensional case.
\begin{align*}
	\sum_{m=1}^{\infty}\frac{1}{m}&\left[\sum_{\substack{i=1\\i\neq m}}^{n}\bar{v}(n,i)\big(4^m-2^{2i-1}\big)\zeta(2m-2i+1)\left(\frac{2n-1}{2}\right)^{2m} - 1\right]\\
	&=\sum_{m=1}^{\infty}\frac{1}{m}\Bigg(\sum_{\substack{i=1\\i\neq m}}^{n}\bar{v}(n,i)\zeta(2m-2i+1)(2n-1)^{2m} \\ 
	&\hspace{5mm} -2^{2i-1} \sum_{\substack{i=1\\i\neq m}}^{n}\bar{v}(n,i)\zeta(2m-2i+1)\left(\frac{2n-1}{2}\right)^{2m}- 1\Bigg)\\
	&=\sum_{m=1}^{\infty}\frac{1}{m}\Bigg(\sum_{\substack{i=1\\i\neq m}}^{n}\bar{v}(n,i)\zeta(2m-2i+1,2n)(2n-1)^{2m}\\ 
	&\hspace{5mm} -2^{2i-1}  \sum_{\substack{i=1\\i\neq m}}^{n}\bar{v}(n,i)\zeta(2m-2i+1,n)\left(\frac{2n-1}{2}\right)^{2m}\Bigg) \\ 
	&\hspace{10mm}+\sum_{m=1}^{\infty}\frac{1}{m}\Bigg[\sum_{\substack{i=1\\i\neq m}}^{n}\bar{v}(n,i)(2n-1)^{2m}\left(\sum_{k=1}^{n-1}(2k)^{-2m+2i-1}-\sum_{k=1}^{2n-1}k^{-2m+2i-1}\right)-1\Bigg]\\
	&=\sum_{i=0}^{n-1}\bar{v}(n,i+1)\Bigg(\sum_{\substack{m=-i\\m\neq 0}}^{\infty}\frac{\zeta(2m+1,2n)}{m+i+1}(2n-1)^{2m+2i+2}\\ 
	&\hspace{5mm} -2^{2i+1}  \sum_{\substack{m=-i\\m\neq 0}}^{\infty}\frac{\zeta(2m+1,n)}{m+i+1}\left(\frac{2n-1}{2}\right)^{2m+2i+2}\Bigg) \\ 
	&\hspace{10mm}+\sum_{m=1}^{\infty}\frac{1}{m}\Bigg[(2n-1)^{2m}\sum_{k=1}^{n-1}(2k-1)^{-2m}\sum_{i=1}^{n}\bar{v}(n,i)(2k-1)^{2i-1} \\
	&\hspace{15mm}+\sum_{i=1}^{n}\bar{v}(n,i)(2n-1)^{2i-1}-1\Bigg] -\left(\frac{H_{n-1}}{2}-H_{2n-1}\right)\sum_{i=1}^{n}\bar{v}(n,i)\frac{(2n-1)^{2i}}{i}\\
	&=\sum_{i=0}^{n-1}\bar{v}(n,i+1)\Bigg(\sum_{\substack{m=-i\\m\neq 0}}^{\infty}\frac{\zeta(2m+1,2n)}{m+i+1}(2n-1)^{2m+2i+2}\\ 
	&\hspace*{5mm} -2^{2i+1}  \sum_{\substack{m=-i\\m\neq 0}}^{\infty}\frac{\zeta(2m+1,n)}{m+i+1}\left(\frac{2n-1}{2}\right)^{2m+2i+2}\Bigg)\eqskip
	&\hspace*{10mm}-\left(\frac{H_{n-1}}{2}-H_{2n-1}\right)\sum_{i=1}^{n}\bar{v}(n,i)\frac{(2n-1)^{2i}}{i}
\end{align*}
where the second step follows from Proposition~\ref{zetaprop} $\mathit{(ii)}$ and, in the last step, we used Proposition~\ref{cfnoddprop} $\mathit{(iii)}$ and $\mathit{(iv)}$ to
conclude the terms inside the square brackets vanish. In order to simplify the two infinite sums inside the parenthesis, consider the second equation in
Proposition~\ref{crazyformulas} $\mathit{(ii)}$ applied twice ($a=2n$, $t=2n-1$ and $a=n$, $t=\frac{2n-1}{2}$ resp.) to obtain
\begin{align*}
	\sum_{\substack{m=-i\\m\neq 0}}^{\infty}\frac{\zeta(2m+1,2n)}{m+i+1}&(2n-1)^{2m+2i+2}\; -2^{2i-1} \sum_{\substack{m=-i\\m\neq 0}}^{\infty}
	\frac{\zeta(2m+1,n)}{m+i+1}\left(\frac{2n-1}{2}\right)^{2m+2i+2}\\
	&=\sum_{k=0}^{2i+1}\binom{2i+1}{k}\Big[\zeta'(-k)+(-1)^{k+1}\zeta'(-k,4n-1)\Big](2n-1)^{2i+1-k}\\
	&\hspace{4mm}- \frac{(2n-1)^{2i+2}}{i+1}\big(\psi(2i+2)-\psi(2n)+\gamma\big)-2\zeta'(-2i-1,2n)\\
	&\hspace{8mm}-2^{2i+1}\Bigg(\sum_{k=0}^{2i+1}\binom{2i+1}{k}\Bigg[\zeta'\left(-k,\frac{1}{2}\right)+(-1)^{k+1}\zeta'\left(-k,2n-\frac{1}{2}\right)\Bigg]\\
	&\hspace{12mm}\times\left(\frac{2n-1}{2}\right)^{2i+1-k}-\frac{(\frac{2n-1}{2})^{2i+2}}{i+1}\big(\psi(2i+2)-\psi(n)+\gamma\big)-2\zeta'(-2i-1,n)\Bigg)\\
	&=\sum_{k=0}^{i}2\binom{2i+1}{2k+1}(2n-1)^{2i-2k}\big(2^{2k+1}\zeta'(-2k-1)-\log(2)\zeta(-2k-1)\big) \\
	&\hspace{4mm} + \sum_{j=1}^{2n-1}\log\left(\frac{2j-1}{2}\right)(2n-2j)^{2i+1}-\sum_{j=1}^{4n-2}\log(j)(2n-j-1)^{2i+1}\\
	&\hspace{8mm}+\frac{(2n-1)^{2i+2}}{i+1}\left( \frac{\psi(2i+2)}{2} -\psi(2n) +\frac{\psi(n)}{2} + \frac{\gamma}{2} \right)\\
	&\hspace{12mm}+(2^{2i+2}-2)\zeta'(-2i-1) + 2\left(\sum_{j=1}^{n-1}\log(j)(2j)^{2i+1}-\sum_{j=1}^{2n-1}\log(j)j^{2i+1}\right)
\end{align*}
where Proposition~\ref{zetaprop} $\mathit{(iii)}$ was used in the last step.
\begin{thm}
	Let $n\in\N$. The determinant of the Laplacian on even-dimensional spheres satisfies the following identity
	\begin{align*}
		-\log\left[ \det\left(\sn{2n}\right)\right] = \zderdim{0}{2n}&=\frac{2}{(2n-1)!}\sum_{k=1}^{n}\zeta'(-2k+1)\big(s(2n-1,2k-1)+s(2n,2k)\big)\\
		&\hspace{16mm}+\sum_{i=1}^{n}\bar{v}(n,i)\frac{(2n-1)^{2i}}{2i}\left(\frac{H_{i-1}}{2}-H_{2i-1}\right)+\log(2n-1)
	\end{align*}
 where $s(n,k)$ are the Stirling numbers of the first kind as defined by equation~\eqref{stirlingnumbers}, and $\bar{v}(n,k)$ are as defined by formula \eqref{cfnoddnorm}.
 \label{zetaevendet}
\end{thm}
\begin{proof}	Using~\eqref{deteven},~\eqref{Cmeven} and the previous result together with~\eqref{detpre2} we obtain
	\begin{align*}
		\zderdim{0}{2n}&=\sum_{i=1}^{n}\bar{v}(n,i)\Bigg[\frac{(2n-1)^{2i}}{4i}H_{i-1} +\frac{(2n-1)^{2i}}{2i}\big(\gamma+\log(4)\big) \eqskip
		&\hspace{5mm}+\sum_{k=1}^{i}2\binom{2i-1}{2k-1}(2n-1)^{2i-2k}\big(2^{2k-1}\zeta'(-2k+1)-\log(2)\zeta(-2k+1)\big)\eqskip
		&\hspace{10mm} +\frac{(2n-1)^{2i+2}}{i+1}\left( \frac{\psi(2i+2)}{2} -\psi(2n) +\frac{\psi(n)}{2} + \frac{\gamma}{2} \right)\eqskip
		&\hspace{15mm}-\left(\frac{H_{n-1}}{2}-H_{2n-1}\right)\frac{(2n-1)^{2i}}{i}\eqskip
		&\hspace{20mm}+\log(4)\zeta(-2i+1)+(2-2^{2i})\zeta'(-2i+1)\eqskip
		&\hspace{25mm} +\sum_{j=1}^{2n-1}\log\left(\frac{2j-1}{2}\right)(2n-2j)^{2i+1}-\sum_{j=1}^{4n-2}\log(j)(2n-j-1)^{2i+1}  \eqskip
		&\hspace{30mm}+2\sum_{j=1}^{n-1}\log(j)(2j)^{2i+1}-2\sum_{j=1}^{2n-1}\log(j)j^{2i+1}\Bigg]-2\log\left(\frac{2}{2n-1}\right)\eqskip
		&=\sum_{i=1}^{n}\bar{v}(n,i)\Bigg[\log(4)\zeta(-2i+1)+\frac{(2n-1)^{2i}}{2i}\left(\log(4)+H_{2i-1}-\frac{H_{i-1}}{2}\right)\eqskip
		&\hspace{5mm}+ \sum_{k=1}^{i}2\binom{2i-1}{2k-1}(2n-1)^{2i-2k}\big(2^{2k-1}\zeta'(-2k+1)-\log(2)\zeta(-2k+1)\big) \eqskip
		&\hspace{10mm}-\log(4)\sum_{j=1}^{n-1}(2j)^{2i-1}\Bigg]-\log\left(\frac{2n-1}{2}\right)\eqskip
		&=\frac{2}{(2n-1)!}\sum_{k=1}^{n}\zeta'(-2k+1)\big(s(2n-1,2k-1)+s(2n,2k)\big)\eqskip
		&\hspace{5mm}+\sum_{i=1}^{n}\bar{v}(n,i)\Bigg[-\log(4)\sum_{k=1}^{i}\binom{2i-1}{2k-1}(2n-1)^{2i-2k}\zeta(-2k+1)+\log(4)\zeta(-2i+1)\eqskip
		&\hspace{10mm}+\frac{(2n-1)^{2i}}{2i}\left(\log(4)+H_{2i-1}-\frac{H_{i-1}}{2}\right)-\log(4)\sum_{j=1}^{n-1}(2j)^{2i-1}\Bigg]+\log\left(\frac{2n-1}{2}\right)\eqskip
		&=\frac{2}{(2n-1)!}\sum_{k=1}^{n}\zeta'(-2k+1)\big(s(2n-1,2k-1)+s(2n,2k)\big)\eqskip
		&\hspace{5mm}+\sum_{i=1}^{n}\bar{v}(n,i)\frac{(2n-1)^{2i}}{2i}\left(\frac{H_{i-1}}{2}-H_{2i-1}\right)+\log(2n-1)
	\end{align*}
	where, in the second step, both \eqref{psiprop} and Proposition \ref{logevenprop} $\mathit{(i)}$ and $\mathit{(ii)}$ are used and, in the third step, we employ Proposition  \ref{cfnoddprop} $\mathit{(v)}$ and finally, in the last step, we use Proposition \ref{zetaprop} $\mathit{(vi)}$ and $\mathit{(vii)}$.
\end{proof}

Using the expressions for the determinant obtained in Theorem \ref{zetaodddet} and Theorem \ref{zetaevendet} we present the first few values explicitly and numerically in the following corollary.
\begin{corollary}\label{lowdimSn}
	The determinant of the Laplacian on $\sn{n}\; (n=2,3,4,5,6,7,8,9)$ are given by
	\begin{align*}
		\det\left(\mathbb{S}^{2}\right)&=e^{\frac{1}{6}}A^4\\ 
		&=3.19531\dots\\
		\det\left(\mathbb{S}^{3}\right)&=\pi\exp\left[\frac{\zeta(3)}{2\pi^2}\right]\\
		&=3.33885\dots\\
		\det\left(\mathbb{S}^{4}\right)&=\frac{1}{3}\exp\left[\frac{83}{144}-\frac{2\zeta'(-3)}{3}\right]A^\frac{13}{3}\\
		&=1.73694\dots\\
		\det\left(\mathbb{S}^{5}\right)&=\frac{\pi}{2}\exp\left[\frac{23\zeta(3)}{24\pi^2}-\frac{\zeta(5)}{8\pi^4}\right]\\
		&=1.76292\dots\\
		\det\left(\mathbb{S}^{6}\right)&=\frac{1}{5}\exp\left[\frac{1381}{2160}-2\zeta'(-3)-\frac{\zeta'(-5)}{30}\right]A^\frac{149}{30}\\
		&=1.29002\dots\\
		\det\left(\mathbb{S}^{7}\right)&=\frac{\pi}{3}\exp\left[\frac{949\zeta(3)}{720\pi^2}-\frac{13\zeta(5)}{24\pi^4}+\frac{\zeta(7)}{32\pi^6}\right]\\
		&=1.22252...\\
		\det\left(\mathbb{S}^{8}\right)&=\frac{1}{7}\exp\left[\frac{4730849}{7257600}-\frac{1199\zeta'(-3)}{360}-\frac{71\zeta'(-5)}{360}-\frac{\zeta'(-7)}{1260}\right]A^\frac{383}{70}\\
		&=1.05041\dots\\
		\det\left(\mathbb{S}^{9}\right)&=\frac{\pi}{4}\exp\left[\frac{16399\zeta(3)}{10080\pi^2}-\frac{2087\zeta(5)}{1920\pi^4}+\frac{31\zeta(7)}{128\pi^6}-\frac{\zeta(9)}{128\pi^8}\right]\\
		&=0.94673\dots
	\end{align*}
	where $A:=e^{\frac{1}{12}-\zeta'(-1)}$ is the Glaisher-Kinkelin's constant.
\end{corollary}
In Appendix~\ref{numerical} we provide the numerical values for determinants up to dimension $100$ for reference. However, the above formulae allow us to compute
much higher dimensions and in Figure~\ref{gr:sphere1} we show a graph with the values up to dimension 10000. The two distinct sets of points correspond to even and odd
dimensions.
\begin{figure}[!ht]
\includegraphics[scale=0.55]{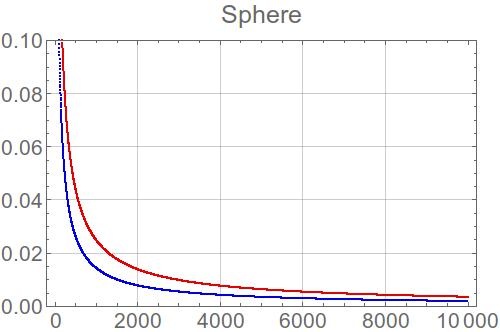}
\caption{\label{gr:sphere1}The first $10\, 000$ values for $\det\left(\sn{n}\right)$ (odd dimensions: blue; even dimensions: red).}
\end{figure}

\subsection{Discrepancies in the numerical values in the literature\label{disc}}
There are some discrepancies in the literature for the values of the determinant of the sphere, namely, between those obtained by Kumagai~\cite{kumagai}
  and Quine and Choi~\cite{quine2}, and those obtained by Choi and Srivastava in~\cite{choisri1}. Commenting upon these discrepancies in~\cite[Section~5.5]{choisri1},
  Choi and Srivastava trace them back to the values used for the coeffiients $c_{-m}$ in these papers, as the differences coincide with some of the values of $c_{-m}$ used.
  They then suggest that the way in which these values are calculated should be examined carefully. The formulae for these coefficients are originally given in
  Voros~\cite[pp. 444]{voros} as in formula~\eqref{Cm} above, and this is what we have used in this paper. We believe that there is a misprint in the simplified
  formula given in~\cite[equation~(3.3)]{voros}, and that the correct formula for positive $m$ should read as
  \[
   c_{-m} = (-1)^m{\ds \lim_{s\to m}} \fr{ \zeta_{\mathfrak{T},M}(s)}{\Gamma(1-s)}.
  \]
To see this, note that following the usual procedure using the Mellin transform of the partition function associated with an eigenvaue sequence $\lambda_{k}$,
namely,
\[
	\Theta(t) = \dsum_{k=0}^{\infty} e^{-t\lambda_k},
\]
we obtain
\[
\begin{array}{lll}
 \zeta_{\mathfrak{T},M}(s) & = & \fr{1}{ \Gamma(s) } \dint_{0}^{+\infty} t^{s-1} \Theta(t)\,{\rm d}t\eqskip
 & = & \fr{1}{ \Gamma(s) } \dint_{0}^{1} t^{s-1} \Theta(t)\,{\rm d}t + \fr{1}{ \Gamma(s) } \dint_{1}^{+\infty} t^{s-1} \Theta(t)\,{\rm d} t\eqskip
 \end{array}
\]
where the rightmost term is analytic in $s$. Following the notation used in~\cite{voros}, the expansion of the partition function at $0^{+}$ for the
Laplace-Beltrami operator on a compact manifold (with or without boundary) may be written as
\[
\Theta(t) \sim \dsum_{k=0}^{+\infty} c_{i_{k}} t^{i_{k}},
\]
where the indexes $i_{k}$ form an increasing sequence of real numbers growing to infinity -- for more details on the coefficients $c_{i_{k}}$ and $i_{k}$ see, for
instance,~\cite{grei} and the references therein. We now sketch the remaining part of the argument, which still
follows a standard approach. Writing the above as
\[
\begin{array}{lll}
 \zeta_{\mathfrak{T},M}(s) & = & \fr{1}{ \Gamma(s) } \dint_{0}^{1} t^{s-1} \left[\Theta(t) - \dsum_{k=0}^{p} c_{i_{k}} t^{i_{k}}\right]\,{\rm d}t + 
 \fr{1}{ \Gamma(s) } \dint_{0}^{1} t^{s-1} \dsum_{k=0}^{p} c_{i_{k}} t^{i_{k}}\,{\rm d}t \eqskip
 & & \hspace*{5mm} +\fr{1}{ \Gamma(s) } \dint_{1}^{+\infty} t^{s-1} \Theta(t)\,{\rm d} t\eqskip
 & = &  \fr{1}{ \Gamma(s) } \dsum_{k=0}^{p} \fr{c_{i_{k}}}{s+i_{k}} + \fr{1}{ \Gamma(s) }F(s),
\end{array}
\]
where $p$ is such that all the singular terms have been incorporated into the finite sum and $F$ is a function analytic for $\re(s)>0$. The above is valid for $\re(s)>\mu$ and, by analytic
continuation, also for $0<\re(s)<\mu$, with the exception of the poles at the points $s=-i_{k}$.
Dividing now both sides by $\Gamma(1-s)$ and using Euler's reflection formula we obtain
\[
 \fr{ \zeta_{\mathfrak{T},M}(s) }{\Gamma(1-s)} = \fr{\sin(\pi s)}{\pi} \dsum_{k=0}^{p} \fr{c_{i_{k}}}{s+i_{k}} +  \fr{\sin(\pi s)}{\pi}F(s).
\]
As was pointed out in~\cite{voros}, the function $\zeta_{\mathfrak{T},M}$ has poles at the positive values of $-i_{k}$ and, whenever this equals an integer $m$, we obtain
\[
 {\ds \lim_{s\to m}} \fr{ \zeta_{\mathfrak{T},M}(s) }{\Gamma(1-s)} =  (-1)^{m} c_{-m},
\]
yielding the desired formula.

\section{Application to other examples\label{otherex}}
The spectra of the Laplacian on the $n$-dimensional hemisphere $\snh{n}$ with Dirichlet boundary conditions, and on the $n$-dimensional real projective space $\snrp{n}$ are 
well known~\cite{bang,berger} and quite similar, for geometrical reasons, to the spectra of spheres. This enables us to replicate
the previous steps for these examples and obtain the recursions, and explicit formulae for the zeta functions and determinants which we indicate below. Since the proofs are
analogous to those for the sphere given above, except in the case of hemispheres for which the results are new, we present them without proof. In addition,
we also revisit the case of the quantum harmonic oscillator, already studied in~\cite{frei}, but for which we now present an explicit expression for the corresponding zeta function.

\subsection{Hemispheres}
The spectral determinant of hemispheres with Dirichlet boundary conditions has never, to the best of our knowledge, been analyzed in the literature before. We shall now
show how the method described in the previous section for the case of spheres may be applied in a similar way to this case. Except for the zero eigenvalue, the spectrum of
the Laplacian on the $n-$dimensional hemisphere with Dirichlet boundary conditions is, in fact, the same as that on the $n-$dimensional sphere, namely, of the form $k(k+n-1)$,
with $k$ a positive integeer. The difference between the two lies in the corresponding multiplicities which are now given by
\[
m_{k}^n = \binom{n+k-2}{k-1},
\]
where $k \in \N$ - see \cite{bang}. As in the sphere case, we apply a shift to the spectrum of $\lambdas{n}$ and the corresponding shifted zeta function is thus defined as
\[
\zsh{s}{\lambdas{n}}=\dsum_{k=1}^{\infty} \dfrac{m_k^n}{\left(k+\tfrac{n-1}{2}\right)^{2s}}
\]
with the series being absolutely convergent for $\re(s)> n$. From this it is possible to derive a recurrence relation similar to Lemma~\ref{zetarecursionlemma} which
now reads as
\begin{equation}\label{recurh}
\zsdimh{s}{\lambdas{n+2}}{n+2}=\dfrac{\zsdimh{s-1}{\lambdas{n}}{n}-\zsdimh{s-\frac{1}{2}}{\lambdas{n}}{n}-
	\left(\frac{n^2-1}{4}\right)\zsdimh{s}{\lambdas{n}}{n}}{n(n+1)}
\end{equation}
valid for all positive integers $n$ and complex numbers $s$ in the domain of the functions involved. Furthermore,
\[
\begin{array}{lll}
	\zsdimh{s}{\lambdas{1}}{1} & = & \zeta(2s)\eqskip
	\zsdimh{s}{\lambdas{2}}{2} & = & (2^{2s-1}-1)\zeta(2s-1) - (2^{2s-1}-\tfrac{1}{2})\zeta(2s)
\end{array}
\]
so that $ \zsdimh{s}{\lambdas{n}}{n}$ may be determined for all $n$ by successive applications of the recursion formula. 
\begin{thm}
	Let $n\in\N$. The zeta function $\zsh{s}{\lambdas{n}}$ satisfies the following identities
	\begin{align*}
	&\zsdimh{s}{\lambdas{2n-1}}{2n-1}=\frac{1}{2}\sum_{i=1}^{n}\bar{u}(n-1,i-1)\bigg[\zeta(2s-2i+2)-(n-1)\zeta(2s-2i+3)\bigg] \\
	&\zsdimh{s}{\lambdas{2n}}{2n}=\sum_{i=1}^{n}\bar{v}(n,i)\bigg[ \left(2^{2s-1}-2^{2i-2}\right)\zeta(2s-2i+1) -(2n-1)\left(2^{2s-1}-2^{2i-3}\right)\zeta(2s-2i+2)\bigg]
	\end{align*}
\label{hemi_rec_sol}
\end{thm}

\begin{proof}\hfill\\- Odd dimensional case: 	
\item $\blacklozenge$ Induction on $n$.\\
\emph{Induction step}:\\ Following the recursion in \eqref{recurh} for $\zsdim{s}{\lambdas{2n+1}}{2n+1}$ we obtain	
\begin{align*}
	&\zsdimh{s}{\lambdas{2n+1}}{2n+1}=\dfrac{\zsdimh{s-1}{\lambdas{2n-1}}{2n-1}-\zsdimh{s-\frac{1}{2}}{\lambdas{2n-1}}{2n-1}-\left(n(n-1)\right)\zsdimh{s}{\lambdas{2n-1}}{2n-1}}{2n(2n-1)}\\ &=\dfrac{1}{2n(2n-1)}\Bigg[\sum\limits_{i=1}^{n}\bar{u}(n-1,i-1)\left(\zeta(2s-2i)-n\zeta(2s-2i+1)\right)\\
	&\hspace{25mm}-(n-1)^2\sum\limits_{i=1}^{n}\bar{u}(n-1,i-1)\left(\zeta(2s-2i+2)-n\zeta(2s-2i+3)\right)\Bigg]\\
	&=\frac{1}{2}\sum_{i=1}^{n}\bar{u}(n-1,i-1)\bigg[\zeta(2s-2i+2)-(n-1)\zeta(2s-2i+3)\bigg]
\end{align*}
where we use the Induction Hypothesis and Proposition~\ref{cfnprop} $\mathit{(iv)}$, $\mathit{(v)}$ and \ref{cfnevenprop} $\mathit{(ii)}$ in that order.
\\
The even dimensional case can be shown in exactly the same way by means of Proposition \ref{cfnoddprop}.
\end{proof}

For completeness, and since the results for hemispheres are new, we will provide some insight into the proof in spite of its similarities with the case of spheres.
We begin by obtaining the recursion~\eqref{recurh} and its solution $\zsh{s}{\lambdas{n}}$, in terms of the Riemann zeta function as presented in Theorem~\ref{det_hemi}
and in total analogy with Sections 2.3 and 2.4. Then we leverage such expressions together with Voros method, namely Corollary \ref{voroscorollary}, to obtain the associated
determinant. Namely, we obtain that $\det\left(\snh{n}\right)$ satisfies the following identity.
\begin{align}
	-\log\left[ \det\left(\snh{n}\right)\right]=\zderdimh{0}{n}&=\zsderh{0}{\lambdas{n}}+\sum_{m=1}^{\lfloor\mu\rfloor}FP\left[\zsh{m}{\lambdas{n}}]\right]\frac{(\lambdas{n})^m}{m}+\sum_{m=2}^{\lfloor\mu\rfloor}c_{-m}H_{m-1}\frac{(\lambdas{n})^m}{m!}\nonumber \\
	&\hspace{10mm}+\sum_{m=\lfloor\mu\rfloor+1}^{\infty} \zsh{m}{\lambdas{n}}\dfrac{(\lambdas{n})^m}{m}
	\label{detprehemi}
\end{align}
\subsubsection{The case of $\snh{2n-1}$ }
We observe that $\zsdimh{s}{\lambdas{2n-1}}{2n-1}$ as given in Theorem \ref{hemi_rec_sol} is convergent in the half-plane $\re(s)> n-1 = \mu$,
and $\lambdas{2n-1} = (n-1) ^2$. The first part of equation \eqref{detprehemi} is given by the following expression.
\begin{equation}
	\zshderdim{0}{\lambdas{2n-1}}{2n-1}=\sum_{i=1}^{n}\bar{u}(n-1,i-1)\big[\zeta'(-2i+2)-(n-1)\zeta'(-2i+3)\big]
	\label{detshemiodd}
\end{equation}
Now consider the coefficients $c^{2n-1}_{-m}$ as defined by~\eqref{Cm} which, in this case, translates into
\begin{align*}
	c^{2n-1}_{-m}&= \res\left(\zsdimh{s}{\lambdas{2n-1}}{2n-1},m\right)\Gamma(m) \eqskip
	&=-\bar{u}(n-1,m)(m-1)! \frac{(n-1)}{4},
\end{align*}
with the third term in~\eqref{detprehemi} then becoming
\begin{equation}
	\sum_{m=2}^{n-1}c^{2n-1}_{-m}H_{m-1}\frac{(\lambdas{2n-1})^m}{m!}=-\sum_{m=2}^{n}\bar{u}(n-1,m)\frac{(n-1)^{2m+1}}{4m}H_{m-1}
	\label{Cm_hemi_odd}.
\end{equation}
Next, we evaluate the finite part of $\zsdimh{m}{\lambdas{2n-1}}{2n-1}$ as indicated in~\eqref{fp}. This gives
\begin{align*}
	FP[\zsdimh{m}{\lambdas{2n-1}}{2n-1}]&=\dfrac{1}{2}\sum_{i=1}^{n}\bar{u}(n-1,i-1)FP[\zeta(2m-2i+2)-(n-1)\zeta(2m-2i+3)]\eqskip
	&= \dfrac{1}{2}\sum_{i=1}^{n}\bar{u}(n-1,i-1)\zeta(2m-2i+2) \eqskip
	&\hspace{5mm} - (n-1) \dfrac{1}{2}\sum_{\substack{i=1\\i\neq m + 1}}^{n}\bar{u}(n-1,i-1)\zeta(2m-2i+3)] \eqskip
	&\hspace{10mm}- \bar{u}(n-1,m) \frac{(n-1)}{2}\gamma
\end{align*}
We proceed with the evaluation of the second and fourth terms in expression~\eqref{detprehemi}, making use of Proposition~\ref{zetaprop} $\mathit{(ii)}$ and
Proposition~\ref{cfnevenprop} $\mathit{(iii)}$ and $\mathit{(iv)}$ in the following way:
\begin{gather*}
 \begin{array}{lll}
	&\dsum_{m=n}^{\infty}\zsdimh{m}{\lambdas{2n-1}}{2n-1}\frac{(\lambdas{2n})^m}{m}+ \dsum_{m=1}^{n-1}FP[\zsdimh{m}{\lambdas{2n-1}}{2n-1}]\frac{(\lambdas{2n-1})^m}{m}
	\eqskip
	&=\dsum_{m=1}^{\infty}\frac{(n-1)^{2m}}{2m}\Bigg[\dsum_{i=1}^{n}\bar{u}(n-1,i-1)\zeta(2m-2i+2) - (n-1) \dsum_{\substack{i=1\\i\neq m+1}}^{n}\bar{u}(n-1,i-1)\zeta(2m-2i+3)\Bigg]\eqskip
	&\hspace{5mm}- \dsum_{i=1}^{n-1}\bar{u}(n-1,i)\frac{(n-1)^{2i+1}}{2i}\gamma	\eqskip
	&=\dfrac{1}{2}\dsum_{i=1}^{n}\bar{u}(n-1,i-1)\left[\dsum_{m=2-i}^{\infty}\dfrac{\zeta(2m,n)}{m+i-1}(n-1)^{2m+2i-2} -(n-1)
	\dsum_{\substack{m=2-i\\m\neq 0}}^{\infty}\dfrac{\zeta(2m+1,n)}{m+i-1} (n-1)^{2m+2i-2}\right]\eqskip
	&\hspace{5mm}- \dsum_{i=1}^{n-1}\bar{u}(n-1,i)\frac{(n-1)^{2i+1}}{2i}(\gamma-2H_{n-1})
\end{array}
\end{gather*}
Once again, to simplify the two infinite sums inside the brackets, consider Proposition \ref{crazyformulas} $\mathit{(i)}$ and $\mathit{(ii)}$ ($a=n$ and $t=n-1$)
to obtain
\begin{align*}
	& \dfrac{1}{2}\sum_{i=1}^{n}\bar{u}(n-1,i-1)\Bigg[\sum_{k=0}^{2i-2}\binom{2i-2}{k}\left[ \zeta'(-k)+(-1)^k\zeta'(-k,2n-1)\right](n-1)^{2i-k-2} \\
	&\hspace{40mm} -\sum_{k=0}^{2i-3}\binom{2i-3}{k}\left[ \zeta'(-k)+(-1)^{k+1}\zeta'(-k,2n-1)\right](n-1)^{2i-k-2} \\
	&\hspace{45mm} + \dfrac{(n-1)^{2i-1}}{i-1}\left[\phi(2i-2)-\phi(n)+\gamma\right] \\
	&\hspace{50mm} -2\zeta'(-2i+2,n) +2(n-1)\zeta'(-2i+3,n)\Bigg]\\
	&\hspace{55mm}- \sum_{i=1}^{n-1}\bar{u}(n-1,i)\frac{(n-1)^{2i+1}}{2i}(\gamma-2H_{n-1})\\
	=&	\sum_{i=1}^{n}\bar{u}(n-1,i-1)\Bigg[\sum_{k=0}^{i-1}\binom{2i-2}{2k}\zeta'(-2k)(n-1)^{2i-2k-2} -\sum_{k=0}^{i-1}\binom{2i-3}{2k-1} \zeta'(-2k+1)(n-1)^{2i-2k-1}\\
	&\hspace{5mm}-\zeta'(-2i+2)+(n-1)\zeta'(-2i+3)	\Bigg] + \sum_{i=1}^{n-1}\bar{u}(n-1,i)\frac{(n-1)^{2i+1}}{2i}H_{2i-1}\\
	=&\frac{1}{(2n-2)!}\sum_{k=0}^{2n-1}\zeta'(-k)[s(2n-1,k+1)+(-1)^k s(2n-2,k)]\\
	&\hspace{5mm}-\sum_{i=1}^{n}\bar{u}(n-1,i-1)\Big[\zeta'(-2i+2)-(n-1)\zeta'(-2i+3)\Big] + \sum_{i=1}^{n-1}\bar{u}(n-1,i)\frac{(n-1)^{2i+1}}{2i}H_{2i-1}.
\end{align*}	
Combining this with \eqref{detshemiodd} and \eqref{Cm_hemi_odd} we are able to compute  \eqref{detprehemi} and obtain $\det\left(\snh{2n-1}\right)$ as given in Theorem \ref{det_hemi}.

\subsubsection{The case of $\snh{2n}$}
We observe that $\zsdimh{s}{\lambdas{2n}}{2n}$ as given in Theorem \ref{hemi_rec_sol} is convergent in the section $\Re(s)> n = \mu$ and $\lambdas{2n-1} = (\tfrac{2n-1}{2}) ^2$. The first part of equation~\eqref{detprehemi} is
given by the following expression.
\[
 \begin{array}{lll}
	\zshderdim{0}{\lambdas{2n}}{2n} & = &\dsum_{i=1}^{n}\bar{v}(n,i)\Big[\log(2)\left(\zeta'(-2i+1)-(2n-1)\zeta'(-2i+2)\right)  \eqskip
	& & \hspace{5mm} + \left(1 - 2^{2i-1}\right)\zeta'(-2i+1)-(2n-1)\left(1-2^{2i-2}\right)\zeta'(-2i+2)\Big].
	\label{detshemievender}
 \end{array}
\]
Now consider the coefficients $c^{2n-1}_{-m}$ as defined by~\eqref{Cm} which, in this case, are given by
\begin{align*}
	c^{2n}_{-m}&= \res\left(\zsdimh{s}{\lambdas{2n}}{2n},m\right)\Gamma(m) \\
	&=2^{2m-3}\bar{v}(n,m)(m-1)!
\end{align*}
and the third part of \eqref{detprehemi} is then
\begin{equation}
	\sum_{m=2}^{n-1}c^{2n}_{-m}H_{m-1}\frac{(\lambdas{2n})^m}{m!}=\sum_{m=2}^{n}\bar{v}(n,m)\frac{(2n-1)^{2m}}{8m}H_{m-1}
	\label{Cm_hemi_even}
\end{equation}
Next we evaluate the finite part of $\zsdimh{m}{\lambdas{2n}}{2n}$. Since 
\begin{align*}
	FP[\zsdimh{m}{\lambdas{2n}}{2n}] &= \sum_{\substack{i=1\\i\neq m}}^{n} \bar{v}(n,i)\left(2^{2m-1}-2^{2i-2}\right)\zeta(2m-2i+1) \\
	& \hspace{5mm} -(2n-1) \sum_{i=1}^{n} \bar{v}(n,i)\left(2^{2m-1}-2^{2i-3}\right)\zeta(2m-2i+2) \\
	& \hspace{10mm} + \bar{v}(n,m) 2^{2m-2} \Big(\gamma + \log(2)\Big)
\end{align*}
we may proceed by evaluating the second and fourth parts of expression~\eqref{detprehemi}  using Proposition~\ref{zetaprop} $\mathit{(ii)}$
to convert the Riemann zeta functions into Hurwitz zeta functions, and Proposition \ref{cfnoddprop} $\mathit{(iii)}$ and $\mathit{(iv)}$ to simplify the resulting extra terms in a similar manner to the even sphere case.
\begin{align*}
	&\sum_{m=n}^{\infty}\zsdimh{m}{\lambdas{2n}}{2n}\frac{(\lambdas{2n})^m}{m}+ \sum_{m=1}^{n-1}FP[\zsdimh{m}{\lambdas{2n}}{2n}]\frac{(\lambdas{2n})^m}{m}=\\
	&=\sum_{m=1}^{\infty}\frac{\left(\tfrac{2n-1}{2}\right)^{2m}}{m}\left[\sum_{\substack{i=1\\i\neq m}}^{n}\bar{v}(n,i)\left(2^{2m-1}-2^{2i-2}\right)\zeta(2m-2i+1)\right.\\
	& \left. \hspace{40mm}- (2n-1) \sum_{i=1}^{n}\bar{v}(n,i)\left(2^{2m-1}-2^{2i-3}\right)\zeta(2m-2i+2)\right]\\
	&\hspace{10mm}+ \sum_{i=1}^{n}\bar{v}(n,i)\frac{(2n-1)^{2i}}{4i}\Big(\gamma + \log(4)\Big)	\\
	&=\sum_{i=1}^{n}\dfrac{\bar{v}(n,i)}{2}\left[\sum_{\substack{m=1-i\\m\neq 0}}^{\infty}\dfrac{\zeta(2m+1,2n)}{m+i}(2n-1)^{2m+2i} -2^{2i-1}\sum_{\substack{m=1-i\\m\neq 0}}^{\infty}\dfrac{\zeta(2m+1,n)}{m+i}\left(\dfrac{2n-1}{2}\right)^{2m+2i}
	\right.\\
	&\left. -(2n-1)\sum_{\substack{m=2-i\\m\neq 0}}^{\infty}\dfrac{\zeta(2m,2n)}{m+i-1}(2n-1)^{2m+2i-2} -(2n-1)2^{2i-2} \sum_{\substack{m=2-i\\m\neq 0}}^{\infty}\dfrac{\zeta(2m,n)}{m+i-1}\left(\dfrac{2n-1}{2}\right)^{2m+2i-2}\right]\\
	&\hspace{10mm}+\sum_{i=1}^{n}\bar{v}(n,i)\frac{(2n-1)^{2i}}{4i}\Big(\gamma + \log(4)+H_{n-1} - 2H_{2n}\Big)
\end{align*}
By substituting $a=2n$ and $t=2n-1$ and $a=n$ and $t=\frac{2n-1}{2}$  into Proposition \ref{crazyformulas} $\mathit{(ii)}$, we can expand and simplify the first and
second infinite sums within the brackets. Likewise, by substituting into
Proposition~\ref{crazyformulas} $\mathit{(i)}$, we can expand and simplify the third and fourth infinite sums.

\begin{gather*}
\begin{array}{llll}
	&\dsum_{i=1}^{n}\dfrac{\bar{v}(n,i)}{2}\Bigg\{\dsum_{k=0}^{2i-1}\binom{2i-1}{k}\left[ \zeta'(-k)+(-1)^{k+1}\zeta'(-k,4n-1)\right](2n-1)^{2i-k-1} \eqskip
	&\hspace{25mm} -\dsum_{k=0}^{2i-2}\binom{2i-2}{k}\left[ \zeta'(-k)+(-1)^{k}\zeta'(-k,4n-1)\right](n-1)^{2i-k-1} \eqskip
	&\hspace{30mm}-2^{2i-1}\dsum_{k=0}^{2i-1}\binom{2i-1}{k}\left[ \zeta'(-k, \tfrac{1}{2})+(-1)^{k+1}\zeta'(-k,2n-\tfrac{1}{2})\right]\left(\dfrac{2n-1}{2}\right)^{2i-k-1} \eqskip
	&\hspace{35mm} +2^{2i-2}(2n-1)\dsum_{k=0}^{2i-1}\binom{2i-1}{k}\left[ \zeta'(-k, \tfrac{1}{2})+(-1)^{k+1}\zeta'(-k,2n-\tfrac{1}{2})\right]\left(\dfrac{2n-1}{2}\right)^{2i-k-2} \eqskip
	&\hspace{40mm} - \dfrac{(2n-1)^{2i}}{i}\left[\phi(2i)-\phi(2n)+\gamma\right] - 2\zeta'(-2i+1,2n) \eqskip
	&\hspace{45mm} + 2^{2i-1}\left[ \dfrac{\left(\tfrac{2n-1}{2}\right)^{2i}}{i}\left[\phi(2i)-\phi(n)+\gamma\right] + 2\zeta'(-2i+1,n) \right]\eqskip
	&\hspace{50mm} + (2n-1) \left[2\zeta'(-2i+2,2n)-2^{2i-1}\zeta'(-2i+2,n)\right] \Bigg\}\eqskip
	&\hspace{55mm} +\dsum_{i=1}^{n}\bar{v}(n,i)\frac{(2n-1)^{2i}}{4i}\Big(\gamma + \log(4) +H_{n-1} - 2H_{2n}\Big)\eqskip
	&=\dsum_{i=1}^{n}\bar{v}(n,i)\Bigg\{\dsum_{k=1}^{i}\left[2^{2k-1}\binom{2i-1}{2k-1}\zeta'(-2k+1)(2n-1)^{2i-2k} - 2^{2k-2}\binom{2i-2}{2k-2}\zeta'(-2k+2)(2n-1)^{2i-2k+1}\right] \eqskip
	&\hspace{25mm} - \log(2)\left(\dsum_{k=0}^{2i-1}(-1)^{k+1}\binom{2i-1}{k}(2n-1)^{2i-k-1}\zeta(-k) +(2n-1)\dsum_{j=1}^{n-1}(2j)^{2i-2} - \dfrac{(2n-1)^{2i}}{2i}\right)\eqskip
	&\hspace{30mm} + \left(2^{2i-1} - 1\right)\zeta'(-2i+1) - (2n-1)\left(2^{2i+2}-1\right)\zeta'(-2i+2)\Bigg\}\eqskip
	&\hspace{40mm} - \dsum_{i=1}^{n}\bar{v}(n,i)\frac{(2n-1)^{2i}}{4i}H_{2i-1}\eqskip
	&=\dfrac{1}{(2n-1)!}\dsum_{k=0}^{n}\zeta'(-k)\left[s(2n,k+1)+ (-1)^{k+1} s(2n-1,k)\right] -  \dsum_{i=1}^{n}\bar{v}(n,i)\frac{(2n-1)^{2i}}{4i}H_{2i-1}\eqskip
	&\hspace{5mm} -\dsum_{i=1}^{n} \bar{v}(n,i)\Bigg\{ \log(2)\left(\dsum_{k=0}^{2i-1}(-1)^{k+1}\binom{2i-1}{k}(2n-1)^{2i-k-1}\zeta(-k) +(2n-1)\dsum_{j=1}^{n-1}(2j)^{2i-2} -
	\dfrac{(2n-1)^{2i}}{2i}\right)\eqskip
	&\hspace{10mm} - \left(2^{2i-1} - 1\right)\zeta'(-2i+1) + (2n-1)\left(2^{2i+2}-1\right)\zeta'(-2i+2)\Bigg\}
\end{array}
\end{gather*}
By using the previous equation along with equations \eqref{Cm_hemi_even} and \eqref{detshemievender} in \eqref{detprehemi}, we can derive the expression for $\det\left(\snh{2n}\right)$ as presented in Theorem \ref{det_hemi}.
We note that the terms multiplied by $\log(2)$ cancel in a manner identical to that in the sphere case, using Proposition~\ref{zetaprop} $\mathit{(vi)}$ and
$\mathit{(vii)}$.
\begin{thm}
	Let $n\in\N$. The determinant of the Laplacian on odd - $\snh{2n-1}$ - and even - $\snh{2n}$ - dimensional hemispheres satisfies the following identity.	
	\[
	\begin{array}{rll}
		-\log\left[ \det\left(\snh{2n-1}\right)\right] & = & \zderdimh{0}{2n-1}\eqskip
		& = & \frac{1}{(2n-2)!}\dsum_{k=0}^{2n-1}\zeta'(-k)\left[s(2n-1,k+1)+(-1)^k s(2n-2,k)\right] \eqskip
		& &  \hspace{10mm}-\dsum_{i=1}^{n-1}\bar{u}(n-1,i)\frac{(n-1)^{2i+1}}{2i}\left(\frac{H_{i-1}}{2}-H_{2i-1}\right)\eqskip
		-\log\left[ \det\left(\snh{2n}\right)\right] & = & \zderdimh{0}{2n}\eqskip
		& = & \frac{1}{(2n-1)!}\dsum_{k=0}^{2n-1}\zeta'(-k)\left[s(2n,k+1)+(-1)^{k+1}s(2n-1,k)\right]\eqskip
		& & \hspace{10mm}+\dsum_{i=1}^{n}\bar{v}(n,i)\frac{\left(2n-1\right)^{2i}}{4i}\left(\frac{H_{i-1}}{2}-H_{2i-1}\right)
	\end{array}
	\]
	\label{det_hemi}
\end{thm}

As in the case of spheres, we give the first $100$ numerical values in Appendix~\ref{numerical}, and show the graphs of the first $10\ 000$ values
in Figure~\ref{gr:hemisphere1}. Note that now the behavior for even and odd dimensions is even more striking, in that the corresponding limits as
the dimension grows are different.
\begin{figure}[!ht]
\includegraphics[scale=0.55]{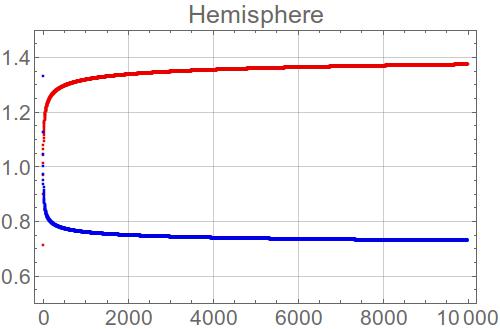}
\caption{\label{gr:hemisphere1}The first $10\, 000$ values for $\det\left(\snh{n}\right)$ (odd dimensions: blue; even dimensions: red).}
\end{figure}

\subsection{Real projective space}

The case of the real projective space has recently been considered by Hartmann and Spreafico in~\cite{HartSpre2}.
Since our approach is different and, above all, the results obtained are given in a different form, we now indicate
the main steps in the application of our method for this case, leaving out the main part of the proofs, as
they are similar to those done above.

The spectrum of the Laplacian is given by integers of the form $2k(2k+n-1)$ with the associated multiplicities
\[
m_{k}^n = \binom{n+2k}{2k}-\binom{n+2k-2}{2k-2}
\]
where $k \in \N_{0}$ - see \cite{berger}. As in the sphere case, we apply a shift on the spectrum of $\lambdas{n}$ and the corresponding shifted zeta function is thus defined as
\[
\zsrp{s}{\lambdas{n}}:=\dsum_{k=1}^{\infty} \dfrac{m_k^n}{\left(2k +\tfrac{n-1}{2}\right)^{2s}}
\]
with the series being absolutely convergent for $\re(s)> n $ -- see Theorem \ref{projectivethm}. From this it is possible to derive a recurrence relation similar to
that given in Lemma~\ref{zetarecursionlemma}, namely,
\begin{equation*}\label{recurrp}
	\zsdimrp{s}{\lambdas{n+4}}{n+4}=\dfrac{\zsdimrp{s-2}{\lambdas{n}}{n}-\left(\tfrac{n^2+1}{2}\right)
		\zsdimrp{s-1}{\lambdas{n}}{n}+\left(\tfrac{ n^2-1}{4}\right)^2 \zsdimrp{s}{\lambdas{n}}{n}}{n(n+1)(n+2) (n+3)}-\left(\dfrac{2}{n+1}\right)^{2 s}
\end{equation*}
valid for all positive integers $n$ and complex numbers $s$ in the domain of the functions involved. Furthermore,
\begin{align*}
	\zsdimrp{s}{\lambdas{1}}{1}&=2^{1-2s}\zeta(2s)\eqskip
	\zsdimrp{s}{\lambdas{2}}{2}&=2^{2-2s}\zeta\left(2s-1,\tfrac{5}{4}\right) \eqskip
	\zsdimrp{s}{\lambdas{3}}{3}&=(1-2^{2-2s})\zeta(2(s-1)) - 1 \eqskip
	\zsdimrp{s}{\lambdas{4}}{4}&= \tfrac{2^{-2s-1}}{3}\left[16\zeta\left(2s-3,\tfrac{7}{4}\right)-\zeta\left(2s-1,\tfrac{7}{4}\right) \right]
\end{align*}
so that $\zsdimrp{s}{\lambdas{n}}{n}$ may be determined for all $n$ by successive applications of the recursion formula. Remarkably, even though the recursion
highlights the need to treat four independent cases, in the end we are able to recover the dimensional dichotomy similar to the previous spaces.
\begin{thm}
	\label{projectivethm}
	For each $n\in\N$ the zeta function $\zsrp{s}{\lambdas{n}}$ satisfies the following identities:
	\[
	\begin{array}{lll}
		\zsdimrp{s}{\lambdas{2n-1}}{2n-1} & = & \dsum_{i=1}^{n}\bar{u}(n-1,i-1)(\tau_n-(-1)^n 2^{2i-2s-2})\zeta(2s-2i+2)-(n-1)^{-2s}\eqskip
		\zsdimrp{s}{\lambdas{2n}}{2n} & = & \dsum_{i=1}^{n}\bar{v}(n,i)2^{4i-2s-2}\zeta(2s-2i+1,\tfrac{5}{4}+\tfrac{\tau_n}{2})-\gamma_n\left[\fr{2}{(2n-1)}\right]^{2s}
	\end{array}
	\]
	where
	\begin{equation*}
		\gamma_n = 
		\Bigg\{
		\begin{array}{ll}
			0 , & \normalfont\text{if } n = 1,2\\
			1 , & \normalfont\text{otherwise}
		\end{array}\qquad\qquad
		\tau_n = 
		\Bigg\{
		\begin{array}{ll}
			0 , & \normalfont\text{if $n$ is odd}\\
			1 , & \normalfont\text{if $n$ is even}
		\end{array}
	\end{equation*}	
	The determinant of the Laplacian on odd-- and even--dimensional real projective spaces,
	$\snrp{2n-1}$ and $\snrp{2n}$, respectively, satisfies		
	\begin{gather*}
	\begin{array}{lll}
		\zderdimrp{0}{2n-1} & = & -\log\left[ \det\left(\snrp{2n-1}\right)\right]\eqskip
		& = & \fr{1}{(2n-2)!}\dsum_{k=0}^{n-1}2^{2k+1}\zeta'(-2k)\left[s(2n-1,2k+1)+s(2n-2,2k)\right] + \log(4(n-1))\eqskip
		\zderdimrp{0}{2n} & = & -\log\left[ \det\left(\snrp{2n}\right)\right]\eqskip
		& = & \fr{1}{(2n-1)!}\dsum_{k=0}^{2n-1}(1-\tau_k 2^{k+1})\zeta'(-k)\left[s(2n,k+1)+s(2n-1,k)\right]\eqskip
		& & \hspace{10mm}+\log(4n-2)+\dsum_{i=1}^{n}\bar{v}(n,i)\frac{\left(2n-1\right)^{2i}}{4i}\left(\frac{H_{i-1}}{2}-H_{2i-1}\right)
	\end{array}
	\end{gather*}
\end{thm}
The first $10\,000$ values are displayed in Figure~\ref{gr:proj}.

\begin{figure}[!ht]
\includegraphics[scale=0.55]{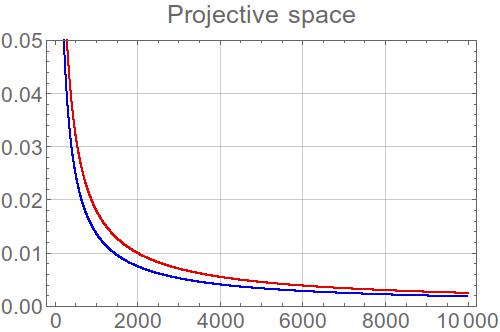}
\caption{\label{gr:proj}The first $10\, 000$ values for $\det\left(\snrp{n}\right)$ (odd dimensions: blue; even dimensions: red).}
\end{figure}

\subsection{The quantum harmonic oscillator revisited}
In~\cite{frei} the first author studied the determinant of the quantum harmonic oscillator in $\R^{n}$, whose spectrum is now given by integers of the form $2k+n$
with the associated multiplicities
\[
 m_{k}^n = \dbinom{n+k-1}{k},
\]
where $k \in \N_{0}$. The corresponding zeta function is thus defined as
\[
 \zH{s}{n} := \dsum_{k=0}^{\infty} \frac{\ds m_{k}}{\ds (2k+n)^s},
\]
with the series being absolutely convergent for $\re(s)>n$ -- see~\cite{frei} for the details. From this it was possible to derive the recurrence
relation~\cite[Theorem~A]{frei}
\begin{equation}\label{recurform}
  \zH{s}{n+2} = \fr{1}{4n(n+1)}\zH{s-2}{n} - \fr{n}{4(n+1)}\zH{s}{n},
\end{equation}
for all positive integers $n$ and complex numbers $s$ in the domain of the functions involved. Furthermore,
\[
\begin{array}{lll}
\zH{s}{1} & = & \left(1-2^{-s}\right)\zeta(s)\eqskip
\zH{s}{2} & = & 2^{-s}\zeta(s-1),
\end{array}
\]
so that $ \zH{s}{n}$ may be determined for all $n$ by successive applications of the recursion formula above. Most of the study
in~\cite{frei} then followed along lines closer to a more classical complex analytic line of approach to the study of zeta functions,
without making use of these formulae. However, the 
similarities with the identities given in Lemma~\ref{zetarecursionlemma} and the other cases considered in this paper are striking, and, indeed, a
similar approach is possible, allowing us to determine an explicit formula for $ \zH{s}{n}$ -- note that the study carried out in~\cite{frei} was mostly concerned with finding, in a rigorous way, the asymptotic behavior of the determinant as $n$ became large.
\begin{thm}
	Let $n\in\N$. The zeta function $\zH{s}{n}$ is given by
	
	\begin{align*}
		\zH{s}{2n}&=\dfrac{2^{-s}}{(2n-1)!}\sum_{i=1}^{n}u(n,i)\zeta(s-2i+1)\eqskip
		\zH{s}{2n-1}&=\dfrac{1}{4^n(2n-2)!}\sum_{i=1}^{n}v(n,i)\left(4-2^{-s+2i}\right)\zeta(s-2i+2)
	\end{align*}
while the corresponding determinants satisfy
\[
-\log\left[\normalfont\text{det}(\Hh{2n})\right]=
\fr{1}{(2 n-1)!}\dsum_{i=1}^n u(n,i)\left[\zeta '(-2i+1)-\log (2) \zeta (-2i+1)\right]
\]
and
\[
-\log\left[\det(\Hh{2n-1})\right]=
 \fr{1}{4^n(2 n-2)!} \dsum _{i=2}^n \left(4-4^{i}\right)v(n,i)  \zeta '(-2i+2)
 +\fr{(-1)^{n}8\log(2)(2n-2)!}{16^n \big((n-1)!\big)^2}. 
\]
\end{thm}

\appendix

\section{Central factorial numbers and relevant properties}
\label{appendixcfn}

The Stirling numbers of  the first kind, $s(n,k)$ can be defined as the coefficients of the expansion of the falling factorial polynomial of degree $n$, denoted $[x]_n$
\begin{equation}\label{stirlingnumbers}
	[x]_0=1,\quad [x]_n=x(x-1)\dots(x-n+1)=:\sum_{k=0}^{n}s(n,k)x^k.
\end{equation}
They have received a lot of attention and are at the core of the study of permutations, as $s(n,k)$ counts the number of permutations of $n$ elements with $k$ cycles. We shall
only present here their properties which are relevant for our purposes, and refer to~\cite[Chapter 5]{comtet} for a more complete treatment of these numbers.

We shall, however, make extensive use of central factorial numbers of the first kind $t(n,k)$, which have received less attention in the literature. They can be defined
as the coefficients of the expansion of the central factorial polynomial of degree $n$, denoted by $x^{[n]}$.
\begin{equation}
	x^{[0]}=1,\quad x^{[n]}=x\left(x+\frac{n}{2}-1\right)\dots\left(x-\frac{n}{2}+1\right)=:\sum_{k=0}^{n}t(n,k)x^k \label{cfngenerating}
\end{equation} 
A short introductory reference to central factorial numbers may be found in~\cite[pp. 213-217]{riordan}, and we refer to~\cite{butzer} for a more systematic treatment,
including their main properties and a variety of applications. Although more closely related to the Stirling numbers of the first kind, the central factorial numbers are also
related to the Euler and Bernoulli numbers (see~\cite{cfneulerbernoulli,cfneulerbernoulli2}). In this appendix, we shall summarize and introduce several properties which were
used throughout the paper.

The identities in the following proposition follow directly from the aforementioned definitions. For more details we refer to \cite[pp. 428, Propostion~2.1]{butzer} with the
exception of entry $\mathit{(vii)}$ which is given in \cite[pp. 480, Propostion~7.3]{butzer}. We remark the similarities between the recursion in entry $\mathit{(i)}$ and the
recursion given for $\zs{s}{\lambdas{n}}$ in Lemma \ref{zetarecursionlemma}.
\begin{proposition} Let $n\in\N$. The following identities hold.
	
	$ \arraycolsep=3pt\def\arraystretch{1.8}
	\begin{array}{cl l}
		\mathit{(i)} & t(n,k)=t(n-2,k-2)-\;\left(\frac{n-2}{2}\right)^2t(n-2,k)  & 2\leq k \leq n\\
		\mathit{(ii)} & s(n,k)=s(n-1,k-1)-\left(n-1\right)s(n-1,k) & 2\leq k \leq n \\	 \mathit{(iii)} & t(n,n)=s(n,n)=1 & n\geq0 \\
		\mathit{(iv)} & t(n,k)=s(n,k)=0 & n < k \\
		\mathit{(v)} & t(n,0)=s(n,0)=\delta_{n,0} & n \geq0\\
		\mathit{(vi)} & t(2n,2k+1)=t(2n+1,2k)=0 & n,k \geq 0 \\
		\mathit{(vii)} & s(n,k)=(-1)^{n+k}\sum\limits_{i=1}^{n}\binom{i-1}{k-1}\left(\frac{n}{2}\right)^{i-k}t(n,k) & n,k \geq 1
	\end{array}
	$
	\label{cfnprop}
\end{proposition}

It will be useful to consider the central factorial numbers with even and odd indices separately, which is also justified by Proposition \ref{cfnprop} $\mathit{(vi)}$.
These are defined as
\begin{align}
	u(n,k)&:=t(2n,2k)\label{cfneven}\\
	v(n,k)&:=4^{n-k}t(2n-1,2k-1)\label{cfnodd}
\end{align}
Both of these integer sequences are present in the On-line Encyclopedia of Integer Sequences (OEIS) under \href{https://oeis.org/A008955}{A008955} for $u(n,k)$ and
under \href{https://oeis.org/A008956}{A008956} for $v(n,k)$. In Table \ref{cfntable} their values for $n,k=0, \dots, 6$ are given. We remark that
$\text{sgn}(s(n,k))=\text{sgn}(v(n,k))=\text{sgn}(u(n,k))=(-1)^{n-k}$ for $n,k\geq 1$ which follows straight from their respective definitions.
\begin{table}[H]
	\caption{Central factorial numbers \\ (Riordan \cite[pp.217, Table 6.1]{riordan})}
	\scalebox{0.8}{
		$
		\begin{array}{|c|rrrrrrr|c|c|rrrrrrr|}
			\cmidrule{1-8} \cmidrule{10-17}
			\multicolumn{8}{|c|}{u(n,k)}&\multicolumn{1}{c}{}&\multicolumn{8}{|c|}{v(n,k)}\\
			\cmidrule{1-8} \cmidrule{10-17}
			k\backslash n&0&1&2&3&4&5&6&&k\backslash n&0&1&2&3&4&5&6  \\
			\cmidrule{1-8} \cmidrule{10-17}
			0&1&0&0&0&0&0&0&&0&1&0&0&0&0&0&0  \\
			1& &1&-1&4&-36&576&-14400&&1& &1&-1&9&-225&11025&-893025\\
			2&&&1&-5&49&-820&21076&&2&&&1&-10&259&-12916&1057221\\
			3&&&&1&-14&273&-7645&&3&&&&1&-35&1974&-172810\\
			4&&&&&1&-30&1023&&4&&&&&1&-84&8778\\
			5&&&&&&1&-55&&5&&&&&&1&-165\\
			6&&&&&&&1&&6&&&&&&&1\\
			\cmidrule{1-8} \cmidrule{10-17}
		\end{array}
		$
	}
	\label{cfntable}
\end{table}
\subsection{Central factorial numbers with even indices}\hfill\\

A detailed treatment of the central factorial numbers with even indices $u(n,k)$ may be found in Shiha \cite{factorialeven}, where a recurrence relation, generating function,
distribution and general formula are presented, together with several combinatorial identities and applications. For our purposes we shall only need the following lemma which
is also a direct consequence of definitions \eqref{cfngenerating} and \eqref{cfneven}.
\begin{lemma}[Shiha {{\cite[pp. 6]{factorialeven}}}]  For $n\in\N$, then 
	\begin{equation*}
		\sum_{k=1}^{n} u(n,k)x^k=\prod_{i=0}^{n-1}(x-i^2)
	\end{equation*}	
	for all $x\in\R$.
	\label{cfnevenlemma}
\end{lemma}
\noindent Consider now the following associated sequence of numbers
\begin{equation}
	\bar{u}(n,k):=\frac{2}{(2n)!}u(n,k)\label{cfnevennorm}
\end{equation}
These are key to understanding the functional determinant of the Laplacian on odd-dimensional spheres and, in particular,
they are intrinsic to the associated spectral zeta function as we shall see in Theorem \ref{zetarecsol}. A few of their properties are
summarized in the following proposition and shall be used several times.
\begin{proposition} Let $n\in\N$. The following identities hold true:
	
	$ \arraycolsep=3pt\def\arraystretch{1.8}
	\begin{array}{cl l}
		\mathit{(i)} & u(n,k)=u(n-1,k-1)-\;\left(n-1\right)^2u(n-1,k) &  2\leq k \leq n \\
		\mathit{(ii)} & \bar{u}(n,k)=\dfrac{\bar{u}(n-1,k-1)-(n-1)^2\bar{u}(n-1,k)}{2n(2n-1)} & 1\leq k \leq n \\
		\mathit{(iii)} & \sum\limits_{k=1}^{n}\bar{u}(n,k)j^{2k}=0 & 0\leq j<n \\
		\mathit{(iv)} & \sum\limits_{k=1}^{n}\bar{u}(n,k)n^{2k}=1 \\
		\mathit{(v)} & \sum\limits_{i=k}^{n}\binom{2i}{2k}n^{2i-2k}\bar{u}(n,i)=\dfrac{s(2n,2k)+s(2n+1,2k+1)}{(2n)!} & 1\leq k \leq n 
	\end{array}
	$
	\label{cfnevenprop}
\end{proposition}

\begin{proof}\hfill\\
	$\mathit{(i)}$ Follows from the definition of $u(n,i)$ in \eqref{cfneven} applied to Proposition \ref{cfnprop} \textit{(i)}.\\
	$\mathit{(ii)}$ Follows from $\mathit{(i)}$ and the definition of $\bar{u}(n,i)$ in \eqref{cfnevennorm}.\\	
	$\mathit{(iii)}$ Considering Lemma \ref{cfnevenlemma} with  $x=j^2$ for $0\leq j<n$  and diving both sides by $\tfrac{(2n)!}{2}$ we obtain
	\begin{equation*}
		\sum_{k=1}^{n}\bar{u}(n,k)j^{2k}=\frac{2}{(2n)!}\prod_{i=0}^{n-1}(j^2-i^2),
	\end{equation*}
	and since $1\leq i, j < n$, the RHS of the previous equation is $0$. \newline
	$\mathit{(iv)}$ Similarly, considering Lemma \ref{cfnevenlemma} with  $x=n^2$ and diving both sides by $\tfrac{(2n)!}{2}$ we obtain
	\begin{equation*}
		\sum_{k=1}^{n}\bar{u}(n,k)n^{2k}=\frac{2}{(2n)!}\prod_{i=0}^{n-1}(n^2-i^2)=1.
	\end{equation*}
	$\mathit{(v)}$ Consider the following
	\begin{align*}
		\sum_{i=k}^{n}\binom{2i}{2k}n^{2i-2k}\bar{u}(n,i) =& 	\frac{2}{(2n)!}\sum_{i=k}^{n}\binom{2i}{2k}n^{2i-2k}t(2n,2i)&&\eqskip
		=& \frac{2}{(2n)!}\sum_{i=2k}^{2n}\binom{i}{2k}n^{i-2k}t(2n,i) && \small\text{Proposition \ref{cfnprop} $\mathit{(vi)}$}\eqskip
		=&\frac{2}{(2n)!}\sum_{i=2k}^{2n}\left[\binom{i-1}{2k-1}+\binom{i-1}{2k}\right]n^{i-2k}t(2n,i)&&\eqskip
		=&\frac{2}{(2n)!} [s(2n,2k) -n \,s(2n,2k+1)] && \small\text{Proposition \ref{cfnprop} $\mathit{(vii)}$}\eqskip
		=&\frac{1}{(2n)!} [s(2n,2k) + s(2n+1,2k+1)] && \small\text{Proposition \ref{cfnprop} $\mathit{(ii)}$}
	\end{align*}
\end{proof}

\subsection{Central factorial numbers with odd indices}\hfill\\

A comprehensive treatment of the central factorial numbers with odd indices $v(n,k)$ may be found in~\cite{factorialodd} and is very similar to the treatment
given for $u(n,k)$ with the exception of a general formula. One notes the unsurprising similarities in what follows. The following lemma is a direct consequence of
definitions \eqref{cfngenerating} and \eqref{cfnodd}.
\begin{lemma} [Zaid, Shiha and El-Desouki {{\cite[pp. 62]{factorialodd}}}]  For $n\in\N$, then 
	\begin{equation*}
		\sum_{k=1}^{n} v(n,k)x^k=x\prod_{i=1}^{n-1}\big(x-(2i-1)^2\big)
	\end{equation*}	
	for all $x\in\R$.
	\label{cfnoddlemma}
\end{lemma}
\noindent Consider now the following associated sequence of numbers
\begin{equation}
	\bar{v}(n,k):= \frac{4^{-(n-1)}}{(2n-1)!}v(n,k)\label{cfnoddnorm}
\end{equation}
As before, these are key to understanding the functional determinant of the Laplacian on even-dimensional spheres and, in particular, they are intrinsic to the associated
spectral zeta function as we shall see in Theorem \ref{zetarecsol}. A few of their properties are presented below in the following proposition which is the counterpart
of Proposition~\ref{cfnevenprop} for even indices, and will also be used throughout the paper.
\begin{proposition} Let $n\in\N$. The following identities hold true.
	
	$ \arraycolsep=3pt\def\arraystretch{2.1}
	\begin{array}{cl l}
		\mathit{(i)} &v(n,k)=v(n-1,k-1)-\;\left(2n-3\right)^2v(n-1,k) & 2\leq k \leq n \\
		\mathit{(ii)} & \bar{v}(n,k)=\dfrac{\bar{v}(n-1,k-1)-(2n-3)^2\bar{v}(n-1,k)}{4(2n-2)(2n-1)} & 1\leq k \leq n \\
		\mathit{(iii)} & \sum\limits_{k=1}^{n}\bar{v}(n,k)(2j-1)^{2k-1}=0 & 0\leq j<n \\
		\mathit{(iv)} & \sum\limits_{k=1}^{n}\bar{v}(n,k)(2n-1)^{2k-1}=1 & \\
		\mathit{(v)} & \sum\limits_{i=k}^{n}\binom{2i-1}{2k-1}(2n-1)^{2i-2k}\bar{v}(n,i)=(s(2n,2k)+s(2n-1,2k-1))\dfrac{2^{-2k+1}}{(2n-1)!} & 1\leq k \leq n \\
	\end{array}
	$
	\label{cfnoddprop}
\end{proposition}

\begin{proof}\hfill\\
	$\mathit{(i)}$ Follows from the definition of $v(n,i)$ in \eqref{cfnodd} applied to Proposition \ref{cfnprop} $\mathit{(i)}$.\\
	$\mathit{(ii)}$ Follows from $\mathit{(i)}$ and the definition of $\bar{v}(n,i)$ in \eqref{cfnoddnorm}.\\
	$\mathit{(iii)}$ Considering Lemma \ref{cfnoddlemma} with  $x=(2j-1)^2$ for $0\leq j<n$ and multiplying both sides by $\tfrac{4^{-(n-1)}}{(2n-1)!(2j-1)}$ we obtain
	\begin{equation*}
		\sum_{k=1}^{n}\bar{v}(n,k)(2j-1)^{2k-1}=\dfrac{4^{-(n-1)}(2j-1)}{(2n-1)!}\prod_{i=1}^{n-1}\big((2j-1)^2-(2i-1)^2\big) 
	\end{equation*}
	and since $1\leq i, j < n$, the RHS of the previous equation is $0$. \newline
	$\mathit{(iv)}$ Similarly, considering Lemma \ref{cfnoddlemma} with  $x=(2n-1)^2$ and multiplying both sides by $\frac{4^{-(n-1)}}{(2n-1)!(2n-1)}$ we obtain
	\begin{equation*}
		\sum_{k=1}^{n}\bar{v}(n,k)(2n-1)^{2k-1}=\frac{4^{-(n-1)}}{(2n-2)!}\prod_{i=1}^{n-1}\big((2n-1)^2-(2i-1)^2\big)=1
	\end{equation*}
	$\mathit{(v)}$ In a similar manner to the proof of Proposition \ref{cfnevenprop} $\mathit{(v)}$, making use of the identities in Proposition \ref{cfnprop} consider the following
	\[
	\begin{array}{lll}
		\dsum_{i=k}^{n}\binom{2i-1}{2k-1}(2n-1)^{2i-2k}\bar{v}(n,i) & = & \fr{2^{2-2k}}{(2n-1)!}\dsum_{i=k}^{n}\binom{2i-1}{2k-1}\left(\frac{2n-1}{2}\right)^{2i-2k}t(2n-1,2i-1)\eqskip
		& = &\fr{2^{2-2k}}{(2n-1)!}\dsum_{i=2k}^{2n}\binom{i}{2k-1}\left(\frac{2n-1}{2}\right)^{i-2k+1}t(2n-1,i)\eqskip
		& = &\fr{2^{2-2k}}{(2n-1)!}\dsum_{i=2k-1}^{2n}\left(\binom{i-1}{2k-1}+\binom{i-1}{2k-2}\right)\eqskip
		& &\hspace*{10mm}\times\left(\fr{2n-1}{2}\right)^{i-2k+1}t(2n-1,i)\eqskip
		& = &\fr{2^{2-2k}}{(2n-1)!} \left[s(2n-1,2k-1) -\left(\fr{2n-1}{2}\right)s(2n-1,2k)\right] \eqskip
		& = &\fr{2^{1-2k}}{(2n-1)!} \big(s(2n-1,2k-1) + s(2n,2k)\big) 
	\end{array}
	\]
\end{proof}
From Proposition \ref{cfnoddprop} we are able to simplify some expressions which are presented below and will prove useful in this essay. They are quite
simple but we provide a proof for clarity.
\begin{proposition} Let $n\in\N$. The following identities hold true.
	\begin{gather*} \arraycolsep=3pt\def\arraystretch{2.2}
	\begin{array}{cl}
		\mathit{(i)}&\sum\limits_{i=1}^{n}\bar{v}(n,i)\left(\sum\limits_{j=1}^{2n-1}\log\left(\frac{2j-1}{2}\right)(2n-2j)^{2i-1}-\sum\limits_{j=1}^{4n-2}\log(j)(2n-j-1)^{2i-1}\right)=\log(4n-2)\\
		\mathit{(ii)}&\sum\limits_{i=1}^{n}\bar{v}(n,i)\left(\sum\limits_{j=1}^{n-1}\log(j)(2j)^{2i-1}-\sum\limits_{j=1}^{2n-1}\log(j)j^{2i-1}\right)=-\log(2)\sum\limits_{i=1}^{n}\sum\limits_{j=1}^{n-1}\bar{v}(n,i)(2j)^{2i-1}-\log(2n-1) \\
	\end{array}
	\end{gather*}
	\label{logevenprop}
\end{proposition}

\begin{proof}
	$\mathit{(i)}$ Note that 
	\begin{align*}
		\sum\limits_{j=1}^{2n-1}&\log\left(\frac{2j-1}{2}\right)(2n-2j)^{2i-1}-\sum\limits_{j=1}^{4n-2}\log(j)(2n-j-1)^{2i-1}=\\&=-\underbrace{\sum\limits_{j=1}^{2n-1}\log(2)(2n-2j)^{2i-1}}_{=0}-\sum\limits_{j=1}^{2n-1}\log(2j)(2n-2j-1)^{2i-1}\\
		&=\log(4n)+\sum\limits_{j=1}^{n-1}\log(2n+2j)(2j+1)^{2i-1}-\sum\limits_{j=1}^{n-1}\log(2n-2j)(2j-1)^{2i-1}\\
		&=\log(4n-2)(2n-1)^{2i-1}-\sum\limits_{j=1}^{n-1}\log\left(\frac{2n-2j}{2n+2j-2}\right)(2j-1)^{2i-1}
	\end{align*}
	where, in the first step, the second sum was separated in its even and odd parts. Now applying $\sum\limits_{i=1}^{n}\bar{v}(n,i)$ on both sides together with Proposition \ref{cfnoddprop} $\mathit{(iii)}$ and $\mathit{(iv)}$ yields the result.\\
	$\mathit{(ii)}$ Note that 
	\[
	\begin{array}{lll}
		\dsum\limits_{j=1}^{n-1}\log(j)(2j)^{2i-1}-\dsum\limits_{j=1}^{2n-1}\log(j)j^{2i-1} & = &-\dsum\limits_{j=1}^{n-1}\log(2)(2j)^{2i-1}-\dsum\limits_{j=1}^{n-1
		}\log(2j+1)(2j+1)^{2i-1}\eqskip	
		& = & -\dsum\limits_{j=1}^{n-1}\log(2)(2j)^{2i-1}-\dsum\limits_{j=1}^{n-2}\log(2j+1)(2j+1)^{2i-1} \eqskip
		& & \hspace*{5mm}- \log(2n-1)(2n-1)^{2i-1} 
	\end{array}
	\]
	and now applying $\sum\limits_{i=1}^{n}\bar{v}(n,i)$ on both sides together with Proposition \ref{cfnoddprop} $\mathit{(iii)}$ and $\mathit{(iv)}$ yields the result.
\end{proof}
\newpage
\section{Numerical values\label{numerical}}

\begin{table}[H]
	\caption{Numerical evaluation of the determinant of the Laplace operator on $\sn{n}$, $\snh{n}$ and $\snrp{n}$ for $n=1,2,\dots,100$.}
	\label{computation:det}
\scalebox{0.82}[0.75]{
\begin{tabular}{|CL|CL||CL|CL||CL|CL|}		
	\toprule[0.5pt]
	\multicolumn{4}{|C||}{ \det(\sn{n})}&\multicolumn{4}{C||}{\text{det}(\snh{n})}&\multicolumn{4}{C|}{\text{det}(\snrp{n})}\\ \midrule
	\midrule	1 & 39.47842  & 2 & 3.195311 & 1 & 6.283185 & 2 & 0.713127 & 1 & 9.869604 & 2 & 2.240353 \\
	3 & 3.338851 & 4 & 1.736943 & 3 & 1.328388 & 4 & 0.896916 & 3 & 2.004050 & 4 & 1.312925 \\
	5 & 1.762919 & 6 & 1.290018 & 5 & 1.126034 & 6 & 0.969512 & 5 & 1.226325 & 6 & 0.985390 \\
	7 & 1.222521 & 8 & 1.050409 & 7 & 1.045956 & 8 & 1.011931 & 7 & 0.909314 & 8 & 0.804377 \\
	9 & 0.946733 & 10 & 0.896183 & 9 & 1.001319 & 10 & 1.040842 & 9 & 0.732550 & 10 & 0.686519\\
	11 & 0.778048 & 12 & 0.786904 & 11 & 0.972171 & 12 & 1.062300 & 11 & 0.618221 & 12 & 0.602558 \\
	13 & 0.663546 & 14 & 0.704655 & 13 & 0.951307 & 14 & 1.079117 & 13 & 0.537518 & 14 & 0.539189 \\
	15 & 0.580375 & 16 & 0.640108 & 15 & 0.935449 & 16 & 1.092804 & 15 & 0.477157 & 16 & 0.489387 \\
	17 & 0.517020 & 18 & 0.587870 & 17 & 0.922876 & 18 & 1.104258 & 17 & 0.430114 & 18 & 0.449052 \\
	19 & 0.467028 & 20 & 0.544581 & 19 & 0.912592 & 20 & 1.114048 & 19 & 0.392302 & 20 & 0.415618 \\
	21 & 0.426495 & 22 & 0.508027 & 21 & 0.903976 & 22 & 1.122559 & 21 & 0.361172 & 22 & 0.387384 \\
	23 & 0.392915 & 24 & 0.476683 & 23 & 0.896618 & 24 & 1.130057 & 23 & 0.335047 & 24 & 0.363179 \\
	25 & 0.364604 & 26 & 0.449464 & 25 & 0.890236 & 26 & 1.136739 & 25 & 0.312775 & 26 & 0.342163 \\
	27 & 0.340384 & 28 & 0.425571 & 27 & 0.884630 & 28 & 1.142750 & 27 & 0.293538 & 28 & 0.323722 \\
	29 & 0.319410 & 30 & 0.404403 & 29 & 0.879653 & 30 & 1.148199 & 29 & 0.276737 & 30 & 0.307391 \\
	31 & 0.301055 & 32 & 0.385501 & 31 & 0.875192 & 32 & 1.153175 & 31 & 0.261923 & 32 & 0.292814 \\
	33 & 0.284847 & 34 & 0.368503 & 33 & 0.871164 & 34 & 1.157746 & 33 & 0.248752 & 34 & 0.279711 \\
	35 & 0.270420 & 36 & 0.353124 & 35 & 0.867501 & 36 & 1.161966 & 35 & 0.236958 & 36 & 0.267861 \\
	37 & 0.257491 & 38 & 0.339133 & 37 & 0.864149 & 38 & 1.165881 & 37 & 0.226330 & 38 & 0.257085 \\
	39 & 0.245830 & 40 & 0.326342 & 39 & 0.861067 & 40 & 1.169528 & 39 & 0.216696 & 40 & 0.247239 \\
	41 & 0.235257 & 42 & 0.314597 & 41 & 0.858219 & 42 & 1.172938 & 41 & 0.207921 & 42 & 0.238201 \\
	43 & 0.225623 & 44 & 0.303770 & 43 & 0.855575 & 44 & 1.176137 & 43 & 0.199889 & 44 & 0.229874 \\
	45 & 0.216803 & 46 & 0.293752 & 45 & 0.853113 & 46 & 1.179148 & 45 & 0.192509 & 46 & 0.222172 \\
	47 & 0.208698 & 48 & 0.284452 & 47 & 0.850811 & 48 & 1.181989 & 47 & 0.185701 & 48 & 0.215026 \\
	49 & 0.201221 & 50 & 0.275793 & 49 & 0.848653 & 50 & 1.184676 & 49 & 0.179400 & 50 & 0.208374 \\
	51 & 0.194300 & 52 & 0.267707 & 51 & 0.846623 & 52 & 1.187225 & 51 & 0.173549 & 52 & 0.202167 \\
	53 & 0.187874 & 54 & 0.260138 & 53 & 0.844708 & 54 & 1.189647 & 53 & 0.168100 & 54 & 0.196357 \\
	55 & 0.181891 & 56 & 0.253035 & 55 & 0.842899 & 56 & 1.191953 & 55 & 0.163012 & 56 & 0.190909 \\
	57 & 0.176305 & 58 & 0.246355 & 57 & 0.841185 & 58 & 1.194152 & 57 & 0.158249 & 58 & 0.185786 \\
	59 & 0.171076 & 60 & 0.240060 & 59 & 0.839558 & 60 & 1.196254 & 59 & 0.153780 & 60 & 0.180960 \\
	61 & 0.166171 & 62 & 0.234115 & 61 & 0.838011 & 62 & 1.198265 & 61 & 0.149578 & 62 & 0.176405 \\
	63 & 0.161561 & 64 & 0.228491 & 63 & 0.836536 & 64 & 1.200193 & 63 & 0.145618 & 64 & 0.172098 \\
	65 & 0.157217 & 66 & 0.223163 & 65 & 0.835129 & 66 & 1.202043 & 65 & 0.141881 & 66 & 0.168018 \\
	67 & 0.153118 & 68 & 0.218106 & 67 & 0.833784 & 68 & 1.203822 & 67 & 0.138346 & 68 & 0.164147 \\
	69 & 0.149243 & 70 & 0.213298 & 69 & 0.832496 & 70 & 1.205533 & 69 & 0.134999 & 70 & 0.160469 \\
	71 & 0.145573 & 72 & 0.208723 & 71 & 0.831261 & 72 & 1.207181 & 71 & 0.131822 & 72 & 0.156970 \\
	73 & 0.142093 & 74 & 0.204362 & 73 & 0.830077 & 74 & 1.208770 & 73 & 0.128805 & 74 & 0.153635 \\
	75 & 0.138786 & 76 & 0.200199 & 75 & 0.828938 & 76 & 1.210304 & 75 & 0.125934 & 76 & 0.150454 \\
	77 & 0.135642 & 78 & 0.196222 & 77 & 0.827842 & 78 & 1.211787 & 77 & 0.123198 & 78 & 0.147416 \\
	79 & 0.132647 & 80 & 0.192418 & 79 & 0.826787 & 80 & 1.213220 & 79 & 0.120589 & 80 & 0.144510 \\
	81 & 0.129791 & 82 & 0.188775 & 81 & 0.825770 & 82 & 1.214608 & 81 & 0.118097 & 82 & 0.141728 \\
	83 & 0.127064 & 84 & 0.185283 & 83 & 0.824788 & 84 & 1.215953 & 83 & 0.115714 & 84 & 0.139063 \\
	85 & 0.124457 & 86 & 0.181932 & 85 & 0.823840 & 86 & 1.217256 & 85 & 0.113434 & 86 & 0.136505 \\
	87 & 0.121963 & 88 & 0.178713 & 87 & 0.822923 & 88 & 1.218521 & 87 & 0.111249 & 88 & 0.134050 \\
	89 & 0.119575 & 90 & 0.175619 & 89 & 0.822035 & 90 & 1.219749 & 89 & 0.109153 & 90 & 0.131690 \\
	91 & 0.117284 & 92 & 0.172642 & 91 & 0.821176 & 92 & 1.220942 & 91 & 0.107142 & 92 & 0.129421 \\
	93 & 0.115086 & 94 & 0.169775 & 93 & 0.820343 & 94 & 1.222102 & 93 & 0.105209 & 94 & 0.127236 \\
	95 & 0.112975 & 96 & 0.167013 & 95 & 0.819536 & 96 & 1.223231 & 95 & 0.103350 & 96 & 0.125131 \\
	97 & 0.110945 & 98 & 0.164348 & 97 & 0.818752 & 98 & 1.224330 & 97 & 0.101562 & 98 & 0.123101 \\
	99 & 0.108992 & 100 & 0.161777 & 99 & 0.817991 & 100 & 1.225400 & 99 & 0.0998387 & 100 & 0.121143 \\
	\bottomrule[0.5pt]	
\end{tabular}
}
\end{table}
\newpage

\section*{Acknowledgements} 
The second author was partially supported by the Funda\c c\~{a}o para a Ci\^{e}ncia e a Tecnologia (Portugal)
through project UIDB/00208/2020.

\end{document}